\documentclass[a4paper,reqno,11pt]{amsart}
\usepackage{mathrsfs, amsthm, amsmath,amsfonts,amssymb}
\usepackage{mathtools}
\usepackage{commath}
\usepackage{fullpage}
\usepackage{xcolor}
\usepackage{enumerate}

\usepackage{mathpazo}
\usepackage[euler-digits]{eulervm}
\usepackage[a4paper, scale={0.72,0.74}, marginratio={1:1}, footskip=7mm, headsep=10mm]{geometry}

\renewcommand{\epsilon}{\varepsilon}

\theoremstyle{plain}
\newtheorem{theorem}{Theorem}[section]
\newtheorem{lemma}[theorem]{Lemma}
\newtheorem{proposition}[theorem]{Proposition}
\newtheorem{corollary}[theorem]{Corollary}

\theoremstyle{definition}

\newtheorem{rem}[theorem]{Remark}

\newenvironment{myitemize}{%
\begin{list}{$\bullet$}%
 	{%
	\setlength{\itemsep}{0.4em}%
	\setlength{\topsep}{0.5em}%
	\setlength\leftmargin{2.45em}%
	\setlength\labelwidth{2.05em}%
	\setlength{\labelsep}{0.4em}%
	}%
	}%
{\end{list}}
\renewenvironment{itemize}{
\begin{myitemize}}%
{\end{myitemize}}

\usepackage{hyperref}

\numberwithin{equation}{section}

\newcommand{\GL}{\mathrm{GL}}
\newcommand{\pos}{\mathcal{P}}
\newcommand{\ort}{\mathcal{O}}
\newcommand{\diff}{\mathop{}\!\mathrm{d}}
\DeclareMathOperator{\tr}{tr}
\newcommand{\RW}{R^{w}_1}
\newcommand{\RWmod}{R^{w}_2}

\newcommand{\mult}{T^{w}}
\newcommand{\Tmod}{\mathbb{T}}
\newcommand{\multmod}{\Tmod^{w}}
\newcommand{\Wish}{\mathrm{W}}
\newcommand{\invWish}{\mathrm{IW}}
\newcommand{\BetaI}{\mathrm{B}^{\mathrm{I}}}
\newcommand{\invBetaI}{\mathrm{IB}^{\mathrm{I}}}
\newcommand{\BetaII}{\mathrm{B}^{\mathrm{II}}}

\usepackage{indentfirst}

\makeatletter
\renewcommand{\@secnumfont}{\bfseries}
\renewcommand\section{\@startsection{section}{1}%
\z@{.7\linespacing\@plus\linespacing}{.5\linespacing}%
{\large\bfseries\scshape\centering}}
\renewcommand\subsection{\@startsection{subsection}{2}%
  \z@{.5\linespacing\@plus.7\linespacing}{-.5em}%
  {\bfseries\scshape}}
\renewcommand\subsubsection{\@startsection{subsubsection}{3}%
  \z@{.5\linespacing\@plus.7\linespacing}{-.5em}%
  {\scshape}}
\makeatother

\author[J.~Arista]{Jonas Arista}
\address{Universit\"at Bielefeld\\
Fakult\"at f\"ur Mathematik \\
Universit\"atsstra{\ss}e 25\\
33615 Bielefeld, Germany}
\email{jarista@math.uni-bielefeld.de}

\author[E.~Bisi]{Elia Bisi}
\address{Technische Universit\"at Wien \\
Institut f\"ur Stochastik und Wirtschaftsmathematik \\
E 105-07 \\
Wiedner Hauptstra{\ss}e 8-10 \\
1040 Wien, Austria}
\email{elia.bisi@tuwien.ac.at}

\author[N.~O'Connell]{Neil O'Connell}
\address{School of Mathematics and Statistics\\
University College Dublin\\
Dublin 4, Ireland}
\email{neil.oconnell@ucd.ie}

\thanks{\emph{}
Research supported by the European Research Council (grant 669306).}



\begin{document}

\title{Matsumoto-Yor and Dufresne type theorems \\ for a random walk on positive definite matrices} 
\maketitle

\begin{abstract}
We establish analogues of the geometric Pitman $2M-X$ theorem of Matsumoto and Yor and of the classical Dufresne identity, for a multiplicative random walk on positive definite matrices with Beta type II distributed increments.
The Dufresne type identity provides another example of a stochastic matrix recursion, as considered by Chamayou and Letac (J.\ Theoret.\ Probab.\ 12, 1999), that admits an explicit solution.
\end{abstract}

\section{Introduction}
\label{sec:intro}

\subsection{Background and literature}
\label{subsec:background}

Let $B=(B(t), t\ge0)$ be a one-dimensional Brownian motion with variance $1/2$ and drift $\alpha\in\mathbb{R}$; namely, $B(t)=W(t)/\sqrt{2} + \alpha t$, where $W(t)$ is a standard one-dimensional Brownian motion.
Let
\[
G(t):= e^{2B(t)} , \qquad 
H(t):= \int_0^t G(s) \diff s, \qquad
t\geq 0.
\]
Matsumoto and Yor~\cite{MY1} proved that
\begin{align}\label{geolifting}
D(t)
:=G(t) H(t)^{-2}, \qquad t>0,
\end{align}
is a diffusion process, in its own filtration, with an explicit infinitesimal generator.

This result is related to the celebrated Pitman $2M-X$ theorem~\cite{Pit, RP}, which states that, setting
\begin{align*}
\mathscr{D}(t)
:= 2 \left(\sup_{0\le s\le t} B(s)\right)- B(t),\qquad t\geq 0,
\end{align*}
the process $(\mathscr{D}(2t), t\geq 0)$ is Markov and equals in law the three-dimensional Bessel process, started at zero, with drift $2|\alpha|$.
The process $D$ can be regarded as a `geometric lifting' of $\mathscr{D}$.
Indeed, the scaling property of Brownian motion and a Laplace approximation argument yield that the process $\big(-\frac{\epsilon}{2}\log D\big(\frac{t}{\epsilon^{2}}\big),t>0\big)$, constructed starting from a Brownian motion $B$ with drift $\epsilon\alpha$, converges weakly to the process $(\mathscr{D}(t),t>0)$ as $\epsilon\downarrow 0$.

A related theorem, usually referred to as Dufresne identity, states that, if $B$ has drift $-\alpha/2$, with $\alpha>0$, then the random variable
\begin{align}\label{DIdentity}
\lim_{t\to\infty} H(t)
= \int_{0}^{\infty} G(s) \diff s
= \int_{0}^{\infty} e^{\sqrt{2}W(s)-\alpha s} \diff s
\end{align}
is finite and has the \emph{inverse gamma distribution} with parameter $\alpha$.
This result has been discovered simultaneously by Dufresne~\cite{D2}, in a financial context, and by Bouchaud, Comtet, Georges and Le Doussal~\cite{BCGLD90}, in a physical context.

Rider and Valk\'o~\cite{RV} proved a matrix version of the Dufresne identity~\eqref{DIdentity} and introduced matrix-valued diffusions that generalise~\eqref{geolifting}.
More precisely, let now $B(t)$ be a $d\times d$ matrix-valued process whose entries evolve as independent one-dimensional Brownian motions with variance $1/2$ and drift $\alpha$.
Let $\pos_{d}$ be the space of $d\times d$ positive definite real symmetric matrices.
Then, following~\cite[Section~2.9]{OC}, a $\GL_{d}$-invariant Brownian motion on $\pos_{d}$ with drift $\alpha$ may be constructed as the process $G=M^{\top} M$, where $M$ solves the Stratonovich SDE $\partial_t M(t) = \big(\partial_t B(t)\big) M(t)$.
It was proved in~\cite{RV} that, when $|\alpha|>\frac{d-1}{2}$, the process
\begin{align}\label{matrixgeolifting}
D(t)
:=H(t)^{-1} \, G(t) \, H(t)^{-1},\qquad t>0,
\end{align}
is a diffusion in $\pos_{d}$, where
\begin{align}
\label{eq:runningIntegral}
H(t):=\int_{0}^{t} G(s) \diff s,\qquad t\ge0.
\end{align}
This result was extended to arbitrary $\alpha\in\mathbb{R}$ in~\cite{OC}.

Moreover, it was shown in~\cite{RV} that, if $G$ has drift $-\alpha/2$, with $\alpha>\frac{d-1}{2}$, then the $d\times d$ random matrix $\lim_{t\to\infty} H(t)$ has the \emph{inverse Wishart distribution} with parameter $\alpha$ (as defined in Section~\ref{subsec:orthInvariant}).
Clearly, in the case $d=1$, the results of~\cite{RV,OC} reduce to the aforementioned scalar versions.
In the physics literature, variations of this matrix Dufresne identity were studied in~\cite{GT20}, in the context of wave scattering in complex media, and in~\cite{GBLD}, in the context of fermions in a Morse potential.

\subsection{Contributions of this work}
\label{subsec:ourContributions}

In this article, we establish analogues of the above results for a certain class of random walks on the space of positive definite matrices.
More precisely, we consider a $\GL_{d}$-invariant multiplicative random walk $R$ on $\pos_{d}$, which may be defined as follows.
Given a (possibly random) initial state $R(0)$ in $\pos_d$ and a sequence $(X(n), n\geq 1)$ of i.i.d.\ and orthogonally invariant random matrices in $\pos_d$, we define recursively
\begin{equation}
\label{eq:RW_intro}
R(n) = R(n-1)^{1/2} X(n) R(n-1)^{1/2} ,
\end{equation}
where, for $x\in\pos_d$, $x^{1/2}$ denotes the only matrix $b\in\pos_d$ such that $bb=x$.
We also consider the running sum of $R$:
\begin{align*}
A(n):=\sum_{k=0}^{n} R(k),\qquad n\ge0,
\end{align*}
which may be viewed as a discrete-time analogue of~\eqref{eq:runningIntegral}.
Then, we define a discrete-time version of~\eqref{matrixgeolifting} by setting
\begin{align*}
S(n) := A(n-1)^{-1} R(n) A(n)^{-1},\qquad n\ge1.
\end{align*}

The first main result of this work is that the process $S=(S(n),n\ge1)$ has the Markov property for a certain choice of the law of the random walk $R$ (see Theorem~\ref{R-P} for a more precise statement). This may be seen as a discrete-time matrix analogue of the Matsumoto-Yor theorem.

\begin{theorem}
\label{thm:MatYor_intro}
If the initial state $R(0)$ of the random walk has the inverse Wishart distribution of parameter $\beta$ and its increments $X(n)$ have the matrix Beta type II distribution (as defined in Section~\ref{subsec:orthInvariant}) of parameters $\alpha$ and $\beta$, then $S=(S(n),n\ge1)$ is a Markov process in its own filtration with an explicit transition kernel.
\end{theorem}

To prove Theorem~\ref{thm:MatYor_intro}, we use a classical criterion (reviewed in Appendix~\ref{app:MarkovFunctions}) for a function of a Markov process to be Markov itself. Notice that $S(n)$ is a (deterministic) function of $(R(n),A(n))$ and that the pair $(R,A)$ has a Markov evolution.
The key step in the proof is to establish an \emph{intertwining relation}: we show that there exists a Markov kernel $\overline{Q}$ such that $\overline{K}\, \Pi =\overline{Q} \, \overline{K}$, where $\Pi$ is the transition kernel of $(R,A)$ and $\overline{K}$ is the intertwining kernel, which encodes the conditional distribution of $(R,A)$ given $S$ (see Corollary~\ref{Maininter}).
This essentially implies that $S$ is Markov with transition kernel $\overline{Q}$.

We stress that Theorem~\ref{thm:MatYor_intro} appears to be new even in the $d=1$ case.
In fact, to arrive at the assumptions of Theorem~\ref{thm:MatYor_intro} and thus, ultimately, at the intertwining relation, we started from some computations and considerations in the scalar case.
These are based on the integrability of a random polymer model with inverse gamma weights, studied in~\cite{COSZ} by means of a geometric Robinson-Schensted-Knuth dynamics.
In particular, the Beta type II random walk of Theorem~\ref{thm:MatYor_intro} arises, essentially, as a ratio of two random walks with inverse gamma increments, see equation~\eqref{eq:randomwalk}.
As we believe that one might both draw motivation and gain insight from these preliminary considerations, we have included an outline of them in Section~\ref{sec:scalarCase}.

Our second main result is the following discrete-time matrix analogue of the Dufresne identity (see Theorem~\ref{Dufresne} for a more precise statement).
\begin{theorem}
\label{thm:Dufresne_intro}
Under the hypotheses of Theorem~\ref{thm:MatYor_intro} and the additional condition $\beta-\alpha>\frac{d-1}{2}$, the limit $\lim_{n\to\infty}A(n)$ exists a.s.\ and has the inverse Wishart distribution with parameter $\beta-\alpha$.
\end{theorem}
We prove this as a consequence of the intertwining relation mentioned above. Essentially, the limiting distribution of $A(n)$ is given by the intertwining kernel $\overline{K}(s; \cdot)$ in the limit as $s\to 0$.

Let us recall that one of the standard proofs of the classical Dufresne identity~\eqref{DIdentity} (see e.g.~\cite{BCGLD90}) relies on constructing a modification of the process $H=(H(t), t\geq 0)$, which has the same fixed-time marginal distributions.
One can show that this auxiliary process is a diffusion (even though $H$ is not!), for which the inverse gamma distribution is stationary.
By letting $t\to\infty$, one obtains the desired distributional identity for $H$.
Close in spirit to this approach, we provide an alternative proof of Theorem~\ref{thm:Dufresne_intro} (see Section~\ref{sec:stochasticEqns}).
Let $\xi$ be the matrix-valued Markov process determined by the stochastic recursion
\begin{equation}
\label{eq:stochRecursion_intro}
\xi(n) := X(n)^{1/2} (I+\xi(n-1)) X(n)^{1/2} , \qquad n\geq 1,
\end{equation}
where $(X(n), n\geq 1)$ are i.i.d.\ Beta type II matrices as above.
By using a different realisation (introduced in Proposition~\ref{prop:invRW}) of the random walk $R$, we show that, whenever $\lim_n A(n)$ is a.s.\ finite, its distribution can be expressed in terms of the unique stationary distribution of $\xi$.
To determine the latter, we are then reduced to solving the stochastic matrix equation
\begin{equation}
\label{eq:stochEq_intro}
Z\stackrel{{\rm d}}{=}
X^{1/2} (I+Z) X^{1/2},
\end{equation}
for $Z$, where $X$ has the same distribution as $X(1)$ and is independent of $Z$. This also establishes a connection with the stochastic matrix recursions considered by Chamayou and Letac in~\cite[Section 6]{CL2}. The case $d=1$ of~\eqref{eq:stochEq_intro} dates back to~\cite[Example 9]{CL}, see also~\cite[Remark 1]{D}.
We solve~\eqref{eq:stochEq_intro} by using certain properties that relate Wishart and matrix Beta distributions (see e.g.~\cite{CaL, OR}), which may be seen as a matrix generalisation of the so-called `beta-gamma algebra' (see Proposition~\ref{SAE}).

For $d=1$, the process defined in~\eqref{eq:stochRecursion_intro}, for general distribution on $X(1)$, is the classical Kesten recursion introduced in~\cite{Kes73}.
This is of fundamental importance in the study of random walks in random environment~\cite{Sol75}, among other mathematics and physics applications.
One of the salient features of the Kesten recursion is that it has a heavy-tailed limit law~\cite{Kes73}.

It is worth comparing~\eqref{eq:stochRecursion_intro} with the matrix stochastic recursion determined by the alternative symmetrisation
\begin{equation}
\label{eq:stochRecursion2_intro}
\xi'(n) := (I+\xi'(n-1))^{1/2}  X(n) (I+\xi'(n-1))^{1/2} , \qquad n\geq 1.
\end{equation}
Although different from $\xi$, the process $\xi'$ has the same stationary distribution (see Corollary~\ref{coro:invariantBothProcesses}).
Gautié, Bouchaud and Le Doussal~\cite{GBLD} considered $\xi'$ for more general distributions on $X(1)$ and studied its limit law in the regime when the matrix dimension $d$ grows large, showing that the distribution of the eigenvalues exhibits heavy tails and thus generalising Kesten's result to the matrix case.
They also considered the continuous limit of the process $\xi'$, which is a diffusion in $\pos_{d}$, finding a version of the matrix Dufresne identity in this setting.
Such a diffusion was also discussed briefly in~\cite[Section 9]{OC}.
On the other hand, a continuous-time version of the process $\xi$, also a diffusion in $\pos_{d}$, was studied in~\cite{RV} (see also~\cite[Section 9]{OC}).
The former is driven by a $\GL_{d}$-invariant Brownian motion on $\pos_{d}$, whereas the latter is driven by a different, $\ort_{d}$-invariant, `Brownian motion' on $\pos_{d}$.
See e.g.~\cite{NRW} for a detailed explanation on the distinction between these two types of Brownian motion on $\pos_{d}$.

\subsection{Related aspects}
\label{subsec:furtherDirections}

A question that naturally arises is whether results similar to those presented here may be obtained for matrix random walks with different distributions.
Somewhat conversely, it would be interesting to investigate if other matrix distributions have representations in terms of matrix random walks.
Notice that other examples of (matrix) stochastic equations have been considered e.g.\ in~\cite{CL,CL2}, which may provide a helpful starting point for answering such questions.
In general, most of these equations involve Wishart and Beta type distributions and generalisations of them.

The Matsumoto-Yor and Dufresne theorems are deeply related to the integrability of the semi-discrete Brownian polymer~\cite{OC12}.
In the matrix setting, certain interacting diffusions with integrable properties have been studied in~\cite{OC}, in connection with the matrix Dufresne identity of~\cite{RV}.
Analogously, given the results of the present work in the discrete matrix setting, one might wonder about possible integrable matrix analogues of discrete polymer models.
In this direction, a system of matrix-valued interacting random walks with inverse Wishart increments, which can be seen as a matrix generalisation of the log-gamma polymer~\cite{Sep12}, has been recently introduced by the authors in~\cite{ABO22}.

\subsection{Organisation of the article}

As a motivation, in Section~\ref{sec:scalarCase} we discuss some of the results of this article in the scalar $d=1$ case.
In Sections~\ref{sec:preliminaries}, we collect some preliminary notions about measures and random walks on positive definite matrices.
In Section~\ref{sec:mainResults}, we establish the two main results of this work: discrete-time matrix versions of the Matsumoto-Yor theorem and of the Dufresne identity.
In Section~\ref{sec:stochasticEqns}, we discuss the connections with stochastic matrix equations and give another proof of the Dufresne type identity.
In Appendix~\ref{app:InvRW}, we prove the equivalence of various constructions of matrix random walks.
Appendix~\ref{app:WishartBeta} collects some properties of the Wishart and matrix Beta distributions.
Appendix~\ref{app:lyapunovExponents} presents a method for computing Lyapunov exponents of matrix random walks using Cholesky decompositions.
Finally, Appendix~\ref{app:MarkovFunctions} reviews the theory of Markov functions.

\section{Motivation: the scalar case}
\label{sec:scalarCase}

In~\cite{COSZ}, a certain Markov dynamics on triangular arrays of positive real numbers, based on the so-called \emph{geometric RSK correspondence}, was considered.
In the particular case where the triangular arrays consist of only two rows, the dynamics is on three `particles' $(X,Y,Z)$ and is defined as follows.
Let $(\mathfrak{A}(n), n\ge1)$ and $(\mathfrak{B}(n), n\ge1)$ be two families of independent \emph{inverse gamma} random variables with parameters $\alpha$ and $\beta$, respectively (for us, the gamma distribution is the continuous distribution on $\mathbb{R}_+$ with density $\Gamma(\alpha)^{-1} x^{\alpha-1} e^{-x}$).
Let us fix a (possibly random) initial state $(X(0),Y(0),Z(0))$ in $\mathbb{R}_{+}^{3}$.
The particle $X$ evolves as a multiplicative random walk on $\mathbb{R}_{+}$ with increments $(\mathfrak{B}(n), n\ge1)$, i.e.\
\begin{align*}
X(n):=X(n-1)\mathfrak{B}(n).
\end{align*}
At time $n$, once $X$ has been updated, the particles $Y$ any $Z$ are updated as follows:
\begin{align*}
Y(n):=\left(Y(n-1)+X(n)\right)\mathfrak{A}(n),\qquad Z(n):=\frac{Z(n-1)}{X(n-1)} \frac{X(n)Y(n-1)}{X(n)+Y(n-1)}.
\end{align*}
The above equations can be seen as a `geometric lifting' of a transformation of paths introduced in~\cite{OY02} (see also~\cite{OC03,OC03B}) in a queueing-theoretic context and closely related to the Pitman $2M-X$ theorem.
By induction, one obtains the following closed expression for $Y$:
\begin{align}
\label{eq:polymer}
Y(n)&= Y(0)\prod_{k=1}^n \mathfrak{A}(k) + X(0)\sum_{k=1}^n \left(\prod_{i=1}^k \mathfrak{B}(i)\right) \left(\prod_{i=k}^n \mathfrak{A}(i)\right) .
\end{align}
Choosing the initial conditions $Y(0)=0$ and $X(0)=1$, $Y(n)$ takes then the form of the (two-row) partition function of the \emph{log-gamma polymer} model, whose distribution was a central object of study in~\cite{COSZ}.

It was proved in~\cite{COSZ} that, for a special class of initial distributions for $(X,Y,Z)$, the process $(Y,Z)$ is (autonomously) Markov with an explicit transition kernel.
An elementary calculation, based on~\eqref{eq:polymer} and the fact that $Y(n)Z(n) = Y(0)Z(0) \prod_{k=1}^n \mathfrak{A}(k)\mathfrak{B}(k)$, shows that the ratio $Z(n)/Y(n-1)$ takes the form
\begin{equation}\label{eq:MYtype}
S(n):=\frac{Z(n)}{Y(n-1)}=\left(\sum_{k=0}^{n-1}R(k)\right)^{-1}R(n)\left(\sum_{k=0}^{n}R(k)\right)^{-1},\quad n\ge1,
\end{equation}
where
\begin{equation}\label{eq:randomwalk}
R(n):=\frac{X(0)}{Z(0)}\prod_{i=1}^{n}\frac{\mathfrak{B}(i)}{\mathfrak{A}(i-1)},\quad n\ge0,
\end{equation}
with the convention $\mathfrak{A}(0):=Y(0)/X(0)$.

The process $S=(S(n),n\ge1)$ in~\eqref{eq:MYtype} may be seen as a discrete-time version of the Matsumoto-Yor diffusion (\ref{geolifting}) for the random walk $R=(R(n),n\ge0)$; however, $S$ is not Markov for the whole class of initial distributions for $(X,Y,Z)$ found by~\cite{COSZ}.
The original motivation for the present article was to find further conditions on $(X(0),Y(0),Z(0))$ that guarantee the Markov property for $S$. 

Define
\begin{equation}
\label{eq:phi_d=1}
\varphi(s):=\int_0^{\infty} x^{\alpha-\beta} (1+sx)^{-\alpha} e^{-1/x} \frac{\diff x}{x} , \qquad s>0.
\end{equation}
Since $S$ is a function of $(Y,Z)$ and we know from~\cite{COSZ} the transition kernel of $(Y,Z)$ explicitly, a few computations (which we omit) show that, if $Z(0):=z$, with $z\in\mathbb{R}_{+}$, and $(X(0),Y(0))$ is chosen with distribution
\begin{align}\label{initialxy}
\frac{1}{\Gamma(\alpha)}\frac{1}{\varphi(z)}\left(\frac{x}{z}\right)^{\alpha-\beta}y^{-\alpha}\exp\left\{-\frac{z}{x}-\frac{1}{y}(1+x)\right\}\frac{\diff x}{x}\frac{\diff y}{y},
\end{align}
then $S$ has the Markov property and its corresponding transition kernel is
\begin{equation*}
\overline{Q}(s;\diff \tilde{s}):=\frac{1}{B(\alpha,\beta)}\frac{\varphi(\tilde{s})}{\varphi(s)}\left(\frac{\tilde{s}}{s}\right)^{\alpha}\left(1+\frac{\tilde{s}}{s}\right)^{-(\alpha+\beta)} e^{-\tilde{s}} \frac{\diff \tilde{s}}{\tilde{s}}.
\end{equation*}
Moreover, using~\eqref{initialxy}, the joint distribution of $(X(0)/Z(0),Y(0)/X(0))$ is
\begin{align}\label{eq:jointRW}
\frac{1}{\Gamma(\alpha)}\frac{z^{-\alpha}}{\varphi(z)} u_{1}^{-\beta}u_{2}^{-\alpha}\exp\left\{-\left(\frac{1}{u_1}+\frac{1}{u_2}+\frac{1}{z}\frac{1}{u_{1}u_{2}}\right)\right\}\frac{\diff u_1}{u_1}\frac{\diff u_2}{u_2}.
\end{align}
Notice that $\varphi(z) \sim z^{-\alpha} \Gamma(\beta)$ as $z\to\infty$.
It is then clear from~\eqref{eq:jointRW} that, under this limit, $X(0)/Z(0)$ and $\mathfrak{A}(0):=Y(0)/X(0)$ in~\eqref{eq:randomwalk} are asymptotically independent inverse gamma variables with parameters $\beta$ and $\alpha$, respectively.
In particular, $R$ will have the law of a multiplicative random walk with inverse gamma initial state of parameter $\beta$ and Beta type II independent increments of parameters $\alpha$ and $\beta$ (the Beta type II distribution of parameters $\alpha$ and $\beta$ can be defined as the law of the ratio between a gamma variable of parameter $\alpha$ and another independent gamma variable of parameter $\beta$).
To sum up, one arrives at the following result:

\begin{theorem}\label{thm:scalar}
Let $R=(R(n),n\ge0)$ be a multiplicative random walk defined as follows: the initial state $R(0)$ has the inverse gamma distribution of parameter $\beta$, and the increments $R(n)/R(n-1)$, $n\geq 1$, are independent random variables with the Beta type II distribution of parameters $\alpha$ and $\beta$.
Then, the process $S(n):=\big(\sum_{k=0}^{n-1}R(k)\big)^{-1}R(n)\big(\sum_{k=0}^{n}R(k)\big)^{-1}$, $n\ge1$, is Markov with transition kernel $\overline{Q}$.
\end{theorem}

In the next sections, we will consider a matrix generalisation of the above result.
In particular, Theorem~\ref{thm:scalar} will follow from Theorem~\ref{R-P} by taking the matrix dimension to be $d=1$.
We will provide a direct proof via some \emph{intertwining relations} and the theory of Markov functions.
The latter `algebraic' approach is convenient because of the non-commutative nature of the variables involved in the matrix setting.

%Consider two independent (multiplicative) random walks $\xi^1$ and $\xi^2$ in $\mathbb{R}^+$ with \emph{inverse gamma increments} of parameters $\beta>0$ and $\alpha>0$, respectively. Define a transformation of paths $(\xi^1,\xi^2)\mapsto (Y,Z)$ by
%\begin{equation}\label{eq:transform}
%Y(n):=\xi^2(n)\sum_{k=0}^{n}\frac{\xi^1(k)}{\xi^2(k-1)},\qquad Z(n):=Z(0)\xi^1(n)\left(\sum_{k=0}^{n}\frac{\xi^1(k)}{\xi^2(k-1)}\right)^{-1},
%\end{equation}
%$n\ge0$, with the convention $\xi^2(-1)=1$. It was proved in~\cite{} that, for special distributions of $(\xi^1(0),\xi^2(0))$, the transformed process $((Y(n),Z(n)),n\ge0)$ is Markov with an explicit transition kernel that we denote by $P$. The transformation of paths~\eqref{eq:transform} is a `positive-temperature' analogue of the transformation introduced in~\cite{} in a queueing-theoretic context. The latter is familiar in the context of the Pitman's $2M-X$ theorem and is closely related to the Robinson-Schensted correspondence (see e.g.~\cite{}).

\section{Preliminaries}
\label{sec:preliminaries}

In this section, we present some facts and notations about positive definite matrices and introduce the matrix measures and random walks that we are interested in.
For background and details, we refer the reader to~\cite{GN, HJ, Ma, T}.
We provide more details on these preliminaries in Appendices~\ref{app:InvRW} and~\ref{app:WishartBeta}.

\subsection{Positive definite matrices}
\label{subsec:matrices}

Let $\pos_d$ denote the set of all $d\times d$ real symmetric positive definite matrices.
%If $x,y\in\pos_{d}$ and $\lambda>0$, then we have:
%\begin{itemize}
%	\item $x$ is invertible and $x^{-1}\in\pos_d$;
%	\item $\lambda x\in\pos_d$ and $x+y\in\pos_d$, but in general $xy\notin\pos_d$ (as it may not be symmetric);
%	\item $y-x\in\pos_d$ if and only if $x^{-1}-y^{-1}\in\pos_d$;	
%\end{itemize}
%For $x\in\pos_d$, there exists a unique $b\in\pos_d$ satisfying $x=bb$; the matrix $b$ is called the \emph{square root} of $x$ and is denoted by $x^{1/2}$.
%For $x,y\in\pos_d$, the product $xy$ has positive eigenvalues, but is not necessarily symmetric, hence it may not belong to $\pos_d$.
%However, the `symmetrised product' $y^{1/2}xy^{1/2}$ does belong to $\pos_d$; see~\eqref{eq:mult} below for generalisations of this.
%We denote the identity matrix by $I$ and the zero matrix by $0$.
For $x\in\pos_{d}$, we denote by $|x|$ the determinant of $x$ and by $\tr[x]$ its trace.
Any $x\in\pos_{d}$ has $d$ (strictly) positive real eigenvalues, so we denote by $\lambda_{\max}(x)$ and $\lambda_{\min}(x)$ its maximum and minimum eigenvalues, respectively.
For arbitrary $x,y\in\pos_{d}$, we have:
\begin{align}
\label{lambdasum}
\lambda_{\min}(x) + \lambda_{\min}(y)
&\leq \lambda_{\min}(x+y)
\leq \lambda_{\max}(x+y)
\leq \lambda_{\max}(x)+\lambda_{\max}(y);\\
\label{lambdaproduct}
\lambda_{\min}(x)\lambda_{\min}(y)
&\leq \lambda_{\min}(xy)
\leq \lambda_{\max}(xy)
\leq\lambda_{\max}(x)\lambda_{\max}(y).
\end{align}

We will consider $\pos_d$ as equipped with the Borel topology, induced by any norm.
Thus, for a sequence $\{x_{n},n\ge0\}\subset\pos_{d}$, we have: $x_{n}$ converges to the zero matrix $0$, as $n\to\infty$, if and only if $\lambda_{\max}(x_{n})\to0$ if and only if $\tr[x_{n}]\to0$.
We will say that $x_{n}\to\infty$, as $n\to\infty$, if and only if $x_{n}^{-1}\to0$ if and only if $\lambda_{\min}(x_{n})\to\infty$.

\subsection{Integration on $\pos_d$}
\label{subsec:integration}

Let $\GL_{d}$ denote the group of $d\times d$ invertible matrices.
The mapping $\GL_d\times \pos_d \to \pos_d$, $(a,x)\mapsto a^{\top} xa$, defines an action of $\GL_{d}$ on $\pos_{d}$, where $a^{\top}$ denotes the transpose of the matrix $a$.
The $\GL_{d}$-\emph{invariant} (under the group action) measure $\mu$ on $\pos_d$ is defined, writing $x=(x_{ij})_{1\le i,j\le d}$, by
\begin{align*}
\mu(\diff x):=|x|^{-(d+1)/2}\prod_{1\le i\le j\le d}\diff x_{ij},
\end{align*}
where $\diff x_{ij}$ is the Lebesgue measure on $\mathbb{R}$.
The measure $\mu$ has also the property that, if $z=x^{-1}$, then $\mu(\diff x)=\mu(\diff z)$.

\subsection{Orthogonally invariant distributions}
\label{subsec:orthInvariant}

Let $\ort_{d}$ be the group of $d\times d$ orthogonal matrices.
We say that a random matrix $X$ in $\pos_{d}$ (equivalently, its distribution) is $\ort_{d}$-\emph{invariant} if $k^{\top}Xk\stackrel{\mathrm{d}}{=}X$, for any $k\in\ort_{d}$, where $\stackrel{\mathrm{d}}{=}$ denotes equality in distribution.

Below we define some families of probability distributions on $\pos_d$, all of which are $\ort_{d}$-invariant: the Wishart and matrix Beta distributions.
In Appendix~\ref{app:WishartBeta} we collect some important relations between them that we use throughout this work.

Let $\alpha>\frac{d-1}{2}$ and let $\Gamma_{d}$ be the $d$-variate gamma function, i.e.\
\begin{align}
\label{eq:gammaFn}
\Gamma_{d}(\alpha):=\int_{\pos_d}|x|^{\alpha}e^{-\tr[x]}\mu(\diff x)=\pi^{\frac{d(d-1)}{4}} \prod_{k=1}^d \Gamma\left(\alpha- \frac{k-1}{2} \right).
\end{align}
We denote by $\Wish_{d}(\alpha)$ the \emph{Wishart distribution} on $\pos_d$ with parameter $\alpha$, defined by
\begin{align}\label{Wishart}
\Gamma_{d}(\alpha)^{-1}|x|^{\alpha}e^{-\tr[x]}\mu(\diff x).
\end{align}
We denote by $\invWish_{d}(\alpha)$ the \emph{inverse Wishart distribution} with parameter $\alpha$, i.e.\
\begin{align}\label{inverseWish}
\Gamma_{d}(\alpha)^{-1}|x|^{-\alpha}e^{-\tr[x^{-1}]}\mu(\diff x).
\end{align}
The distribution~\eqref{Wishart} is also known as (matrix) \emph{gamma distribution} since, when $d=1$, it defines the (real) gamma distribution with shape parameter $\alpha$ and scale parameter $1$.  

For $\alpha>\frac{d-1}{2}$ and $\beta>\frac{d-1}{2}$, let $B_d$ be the $d$-variate beta function
\begin{align}
\label{eq:betaFn}
B_d(\alpha,\beta)
:=\frac{\Gamma_{d}(\alpha)\Gamma_{d}(\beta)}{\Gamma_{d}(\alpha+\beta)}
=\int_{\pos_d}|x|^{\alpha}|I+x|^{-(\alpha+\beta)} \mu(\diff x),
\end{align}
where $I$ is (as always from now on) the $d\times d$ identity matrix.
We denote by $\BetaI_{d}(\alpha,\beta)$ the (matrix) \emph{Beta type I distribution} on $\pos_{d}$ with parameters $\alpha$ and $\beta$, which is
\begin{align}\label{Beta}
B_d(\alpha,\beta)^{-1}|x|^{\alpha}|I-x|^{\beta-(d+1)/2}\mathbb{1}_{\pos_{d}}(I-x)\mu(\diff x).
\end{align}
We denote by $\invBetaI_d(\alpha,\beta)$ the (matrix) \emph{inverse Beta type I distribution}, i.e.\ the distribution of a matrix $X$ such that $X^{-1}$ has the $\BetaI_{d}(\alpha,\beta)$ distribution.
Finally, we denote by $\BetaII_{d}(\alpha,\beta)$ the (matrix) \emph{Beta type II distribution} with parameters $\alpha$ and $\beta$, which is 
\begin{align}\label{BetatypeII}
B_d(\alpha,\beta)^{-1}|x|^{\alpha}|I+x|^{-(\alpha+\beta)}\mu(\diff x).
\end{align}
The latter is also known as (matrix) \emph{Beta prime distribution}.

\subsection{Invariant random walks on $\pos_{d}$}
\label{subsec:InvRW}

We now deal with $\GL_d$-invariant multiplicative random walks on positive definite matrices with $\ort_d$-invariant increments.
We provide various constructions, which are suited to our purposes.
The equivalence of such constructions is stated in Proposition~\ref{prop:invRW}, which, to the best of our knowledge, cannot be found in the literature as such; we therefore provide a proof of it in Appendix~\ref{app:InvRW}.
On the other hand, for more comprehensive accounts of the literature on products of random matrices, we refer the reader to standard texts such as~\cite{BL, Furman}; see also~\cite{CL2}.

Notice that, for $x,y\in\pos_d$, the usual matrix product $xy$ has positive eigenvalues, but is not necessarily symmetric, hence it may not belong to $\pos_d$.
We will define a sort of 'symmetrised product' that lives in $\pos_d$, recalling from Section~\ref{subsec:integration} that $a^\top x a \in\pos_d$ for all $x\in\pos_d$ and $a\in\GL_d$.
Let $w$ be any measurable function $w\colon \pos_d \to \GL_d$ such that
\begin{equation}
\label{eq:splitFn}
x = w(x)^{\top} w(x)
\qquad\quad
\text{for all } x\in \pos_d.
\end{equation}
Notable examples of such functions are:
\begin{itemize}
\item the \emph{square root function} $x \mapsto b$, which maps $x$ to the unique $b\in\pos_d$ such that $bb=x$ ($b$ is called the \emph{square root} of $x$ and denoted by $x^{1/2}$);
\item the \emph{Cholesky function} $x \mapsto u$, where $u$ is the unique upper triangular matrix with positive diagonal entries such that $x=u^{\top} u$ (the latter equality is usually referred to as the \emph{Cholesky decomposition} of $x$).
\end{itemize}
For any choice of the function $w$ satisfying~\eqref{eq:splitFn} and for any $y\in\pos_d$, we may now define a `symmetrised product operation' by $y$ in $\pos_d$ as
\begin{equation}
\label{eq:mult}
\mult_y\colon \pos_d \to \pos_d ,
\qquad
\mult_y(x) := w(y)^{\top} x\, w(y), \qquad x\in\pos_{d}.
\end{equation}

Let $Y=(Y(n), n\geq 0)$ be a (not necessarily time-homogeneous) Markov process in $\pos_{d}$.
We will say that $Y$ is \emph{$\GL_d$-invariant} if, for every $n$, the time-$n$ transition kernel of the process $(a^{\top}Y(n)a, n\geq 0)$ does not depend on the choice of $a\in\GL_d$ (the initial distribution may, however, depend on $a$).
In Appendix~\ref{app:InvRW} we will prove:

\begin{proposition}
\label{prop:invRW}
Let $(X(n), n\ge 1)$ be a family of independent and $\ort_{d}$-invariant random matrices in $\pos_d$ and let $M$ be any random matrix in $\pos_d$ that is independent of $(X(n), n\ge 1)$.
Define $\RW=(\RW(n), n\geq 0)$ and $\RWmod = (\RWmod(n), n\geq 0)$ by
\begin{align}
\label{eq:RW}
\RW(n) &:=\mult_{\RW(n-1)}(X(n)), \quad n\ge1, \quad\text{with}\quad \RW(0):=M,\\
\label{eq:RWmod}
\RWmod(n)
&:= \mult_{M}\circ \mult_{X(1)}\circ\dots\circ \mult_{X(n)}(I), \quad n\ge0 .
\end{align}
Then, $\RW$ and $\RWmod$ are $\GL_d$-invariant Markov processes with the same law.
Moreover, their law is the same for any choice of the function $w$ satisfying~\eqref{eq:splitFn}.
\end{proposition}

\begin{rem}
Notice that $\RWmod$ has a closed form, while $\RW$, in general, does not.
For instance, if $w$ is the square root function $w(x):=x^{1/2}$, then the processes $\RW$ and $\RWmod$ are given by
\begin{align*}
\RW(n) &:=\RW(n-1)^{1/2} X(n) \RW(n-1)^{1/2},\quad n\ge1, \quad\text{with}\quad \RW(0):=M, \\
\RWmod(n)
&:= M^{1/2} X(1)^{1/2}\cdots X(n-1)^{1/2} X(n) X(n-1)^{1/2} \cdots X(1)^{1/2}  M^{1/2},\quad n\ge0 .
\end{align*}
\end{rem}

\begin{rem}
In the commutative case $d=1$, no matter the choice of $w$, we have
\[
\RW(n)=\RWmod(n)=M\cdot X(1)\cdots X(n) .
\]
\end{rem}
 
Throughout this article, a \emph{$\GL_d$-invariant random walk on $\pos_d$} with initial state $M$ and increments $(X(n), n\ge 1)$ will be referred to as any stochastic process whose law is determined by Proposition~\ref{prop:invRW}.
Each of the two basic constructions, \eqref{eq:RW} and~\eqref{eq:RWmod}, can be based on any multiplication operation of type~\eqref{eq:mult} and thus, ultimately, on any function $w$ satisfying~\eqref{eq:splitFn}.
In the special case in which $w$ is the Cholesky function, the Markov processes $\RW$ and $\RWmod$ are not only identical in law, but even almost surely; this fact can be easily verified by induction, using the uniqueness of the Cholesky decomposition.

All the results of Sections~\ref{sec:mainResults} and~\ref{sec:stochasticEqns} are of a distributional nature, hence they hold for any of the constructions of Proposition~\ref{prop:invRW}.
Only the \emph{proof methods} of Section~\ref{sec:stochasticEqns}, where the connection with stochastic matrix recursions is discussed, will be explicitly based on the second construction $\RWmod$.
Finally, in Appendix~\ref{app:lyapunovExponents}, some explicit computations of so-called Lyapunov exponents simplify greatly when using the Cholesky function, even though the results we present do hold in full generality.

\section{Main results}
\label{sec:mainResults}

%In this section, we use the theory of Markov functions (reviewed in Appendix~\ref{app:MarkovFunctions}) to give a probabilistic interpretation of the intertwining relation of Corollary~\ref{Maininter}.

In this section, we establish a Matsumoto-Yor type theorem for a multiplicative random walk on $\pos_d$ with Beta type II distributed increments (Theorem~\ref{R-P}) and also prove an analogue of the classical Dufresne identity (Theorem~\ref{Dufresne}).

Throughout, we consider kernels between measurable spaces, which can be equivalently viewed as integral operators.
We refer to Appendix~\ref{app:MarkovFunctions} for standard notions and notational conventions about kernels.
The spaces we consider will be usually (subsets of) $\pos_{d}^{n}$, the $n$-ary cartesian power of $\pos_{d}$, with their Borel sigma-algebras.
We denote by $\delta(x; \diff y)$ the Dirac measure at $x$.

\subsection{Matsumoto-Yor type theorem}
Let $\alpha>\frac{d-1}{2}$ and $\beta>\frac{d-1}{2}$. Let $(X(n), n\geq1)$ be a family of independent and identically distributed (i.i.d.) random matrices in $\pos_{d}$ with the Beta type II distribution $\BetaII_{d}(\alpha,\beta)$. 

Consider a $\GL_{d}$-invariant random walk $R=(R(n), n\ge0)$ on $\pos_{d}$ with increments $(X(n),n\ge1)$ and initial state $R(0)$, independent of $(X(n), n\geq1)$ and with the inverse Wishart $\invWish_{d}(\beta)$ distribution.
The transition kernel of $R$ will then be
\begin{equation}
\label{kernelofR}
P(r;\diff \tilde{r}):=B_d(\alpha,\beta)^{-1}|r^{-1}\tilde{r}|^{\alpha}|I+r^{-1}\tilde{r}|^{-(\alpha+\beta)}\mu(\diff \tilde{r}), \qquad r,\tilde{r}\in\pos_d.
\end{equation}
Define also the sub-Markov (or killed) kernel
\begin{equation}
\label{kernelQ}
Q(s;\diff \tilde{s}):= P(s;\diff \tilde{s}) e^{-\tr[\tilde{s}]}, \qquad s,\tilde{s}\in\pos_d.
\end{equation}

Consider now the running sum of $R$, i.e.\
\begin{align}\label{runningsum}
A(n):=\sum_{k=0}^{n}R(k),\qquad n\ge0,
\end{align}
and define, in analogy to~\eqref{eq:MYtype}, the `Matsumoto-Yor type process'
\begin{align}\label{functional}
S(n):=A(n-1)^{-1}R(n)A(n)^{-1},\quad n\geq1.
\end{align}
The first main result of this work (see Theorem~\ref{thm:MatYor_intro} below) states that the process $S=(S(n),n\ge1)$ is Markov.

We first need to define additional kernels, whose interpretation in terms of $R$, $A$ and $S$ will be shortly given in Remark~\ref{conditionaldistribution}:
\begin{align}\label{newkernelK}
k(s;\diff a)&:=|a|^{\alpha-\beta}|I+sa|^{-\alpha}e^{-\tr[a^{-1}]}\mu(\diff a), && s,a\in\pos_d, \\
\label{kernelK}
K(s; \diff r \diff a)&:=k(s;\diff a)\,\delta(a(s^{-1}+a)^{-1}a; \diff r), && s,r,a\in\pos_d.
\end{align}
Their (common) normalisation
\begin{equation}\label{functionphi}
\varphi(s):=\int_{\pos_d}k(s;\diff a)
=\int_{\pos_d^2} K(s; \diff r \diff a)
\end{equation}
is finite for any $s\in\pos_d$ and reduces to~\eqref{eq:phi_d=1} in the case $d=1$.

\begin{rem}
The integral $\varphi$ is actually finite under the more general hypotheses $\alpha\in\mathbb{R}$ and $\beta>\frac{d-1}{2}$.
In fact, for any such choice of parameters, we have
\begin{align*}
\varphi(s)=|s|^{\beta-\alpha}\,\Psi(\beta,\beta-\alpha+(d+1)/2;s),
\end{align*}
where $\Psi$ is the confluent hypergeometric function of type II of matrix argument (see~\cite[Definition 1.6.3]{GN}).
Moreover, we can write
\begin{align*}
\varphi(s)=|s|^{\frac{1}{2}(\beta-\alpha-\frac{d+1}{2})}e^{\frac{1}{2}\tr[s]}\,W_{\frac{d+1}{4}-\frac{\alpha+\beta}{2},\frac{\beta-\alpha}{2}}(s),
\end{align*}
where $W$ is the classical Whittaker function of matrix argument (see~\cite[Section 5.2.5]{Ma}).
\end{rem}

We now define the following modifications of the kernels $K$ and $Q$:
\begin{align}
\label{eq:normalisedKernels}
\overline{K}(s;\diff r \diff a):=\frac{1}{\varphi(s)} K(s; \diff r \diff a), \qquad\quad
\overline{Q}(s; \diff \tilde{s}):=\frac{\varphi(\tilde{s})}{\varphi(s)} Q(s;\diff \tilde{s}).
\end{align}
By~\eqref{functionphi}, $\overline{K}$ is simply the normalised version of $K$.
On the other hand, $\overline{Q}$ will turn out to be also Markov (in fact, by Corollary~\ref{eigenequa} below, $\overline{Q}$ will be the Doob transform of the sub-Markov kernel $Q$ through the $Q$-harmonic function $\varphi$).

Finally, define the measure
\begin{align}
\label{kerinf2}
\eta(\diff \tilde{s})&:=B_d(\alpha,\beta)^{-1}\Gamma_{d}(\beta)^{-1}\varphi(\tilde{s})\,|\tilde{s}|^{\alpha}e^{-\tr[\tilde{s}]}\mu(\diff \tilde{s}) .
\end{align}
It is not difficult to check that the density of $\eta$ with respect to $\mu$ is the pointwise limit of the density of $\overline{Q}(s; \cdot)$ as $s\to\infty$.
We will show in Lemma~\ref{lemma:etaProbability} that there is no loss of probability mass under this limiting procedure.

\begin{theorem}\label{R-P}
The process $S=(S(n),n\ge1)$ defined  in~\eqref{functional} is a time-homogeneous Markov process (in its own filtration) in $\pos_d$ with initial distribution $\eta$ and transition kernel $\overline{Q}$.
Moreover, for any bounded measurable function $f:\pos_d^{2}\to\mathbb{R}$ and $n\ge1$, we have
\begin{align}
\label{eq:R-P_conditional}
\mathbb{E}\left[f(R(n),A(n)) \mid S(1),\dots ,S(n-1),S(n)\right]=\overline{K}\,f (S(n))\qquad\text{a.s.}
\end{align}
\end{theorem}

\begin{rem}\label{conditionaldistribution}
Essentially, \eqref{eq:R-P_conditional} states that $\overline{K}$ should be interpreted as the conditional distribution of $(R(n), A(n))$ given $S(n)$, for any $n\ge1$.
\end{rem}

We will establish Theorem~\ref{R-P} via a classical criterion (reviewed in Appendix~\ref{app:MarkovFunctions}) for a function of a Markov process to be Markov itself.
Using~\eqref{runningsum} and~\eqref{functional}, the process $S$ may be written as
\begin{align}
S(n)=\phi(R(n),A(n)),\qquad n\ge1,
\end{align}
where $\phi$ is the function
\begin{equation}
\label{eq:MarkovFunction}
\phi\colon D\to \pos_d, \qquad 
\phi(r,a):=(a-r)^{-1}ra^{-1}
\end{equation}
on the domain $D:=\{(r,a)\in\pos_d^{2}: a-r\in\pos_d\}$.
Notice that $((R(n),A(n)),n\ge0)$ is a Markov process in $\pos_d^{2}$ with transition kernel
\begin{align}
\label{kernelPi}
\Pi(r,a;\diff \tilde{r}\diff \tilde{a})&:=P(r;\diff \tilde{r})\,\delta(a+\tilde{r}; \diff \tilde{a})
\end{align}
and initial distribution
\begin{align}
\label{kerinf}
\lambda(\diff r \diff a)&:=\Gamma_{d}(\beta)^{-1}|a|^{-\beta}e^{-\tr[a^{-1}]}\,\mu(\diff a)\,\delta(a;\diff r),
\end{align}
since $A(0)=R(0)\sim \invWish_{d}(\beta)$.
The key step in the proof is to establish an \emph{intertwining relation} between the kernels $\Pi$ and $Q$ through the `intertwining kernel' $K$:
\begin{proposition}\label{Int}
The following intertwining relation holds:
\begin{align}\label{intertw2}
K\,\Pi=Q\,K,
\end{align} 
where both sides are kernels from $\pos_d$ to $\pos_d^{2}$.
\end{proposition}

\begin{proof}
The proof is by direct calculation of both sides of~\eqref{intertw2}.
Using~\eqref{kernelK} and~\eqref{kernelPi}, for any bounded measurable function $f$ on $\pos_{d}^{2}$ we have
\begin{align*}
K\,\Pi\,f\,(s)&=\int_{\pos_{d}^{2}} K(s;\diff r \diff a)\,\int_{\pos_{d}^{2}}\Pi(r,a; \diff \tilde{r} \diff \tilde{a})\, f(\tilde{r},\tilde{a})\\
&=\int_{\pos_d}k(s;\diff a)\,\int_{\pos_{d}}P(a(s^{-1}+a)^{-1}a;\diff \tilde{r})\,f(\tilde{r},a+\tilde{r}).
\end{align*}
Note that both kernels $k$ and $P$ are absolutely continuous with respect to the reference measure $\mu$.
Thus, writing the corresponding densities and interchanging the order of integration, we have
\begin{align*}
K\,\Pi\,f\,(s)=B_d(\alpha,\beta)^{-1}&\int_{\pos_{d}}\mu(\diff \tilde{r})\,\int_{\pos_{d}}\mu(\diff a)\,|a|^{\alpha-\beta}\left|I+sa\right|^{-\alpha}e^{-\tr[a^{-1}]}\\
&\left|a^{-1}(s^{-1}+a)a^{-1}\tilde{r}\right|^{\alpha}\left|I+a^{-1}(s^{-1}+a)a^{-1}\tilde{r}\right|^{-(\alpha+\beta)}f(\tilde{r},a+\tilde{r}).
\end{align*}
Substituting the variable $a$ with $\tilde{a}:=a+\tilde{r}\in\pos_d$ and using the fact that
\begin{align*}
|a|^{(d+1)/2}\mu(\diff a)=|\tilde{a}|^{(d+1)/2}\mu(d \tilde{a}),
\end{align*}
after some algebraic manipulations we obtain
\begin{equation}
\label{eq:leftSide}
\begin{aligned}
K\,\Pi\,f\,(s)=B_d(\alpha,\beta)^{-1}|s|^{\beta}\int_{\pos_{d}}&\mu(\diff \tilde{r})\,|\tilde{r}|^{\alpha}\int_{\pos_{d}}\mu(\diff \tilde{a})\,\mathbb{1}_{\pos_d}(\tilde{a}-\tilde{r})\,e^{-\tr[(\tilde{a}-\tilde{r})^{-1}]}\\
&|s\tilde{a}+(\tilde{a}-\tilde{r})^{-1}\tilde{r}|^{-(\alpha+\beta)}|(\tilde{a}-\tilde{r})^{-1}\tilde{a}|^{(d+1)/2}\,f(\tilde{r},\tilde{a}).
\end{aligned}
\end{equation}

On the other side, again by~\eqref{kernelK}, we have
\begin{align*}
Q\, K\,f\,(s)=\int_{\pos_d}Q(s; \diff \tilde{s})\,\int_{\pos_{d}}k(\tilde{s};\diff \tilde{a})\,f(\tilde{a}(\tilde{s}^{-1}+\tilde{a})^{-1}\tilde{a},\tilde{a}).
\end{align*}
Similarly to the previous case, we write the densities of $Q$ and $k$ and then interchange the order of integration, obtaining
\begin{align}\label{eq:rightSide}
\begin{aligned}
Q\, K\,f\,(s)
= \, &B_d(\alpha,\beta)^{-1} \int_{\pos_d} \mu(\diff \tilde{a}) |\tilde{a}|^{\alpha-\beta} e^{-\tr[\tilde{a}^{-1}]}
\int_{\pos_d} \mu(\diff \tilde{s}) \left|s^{-1} \tilde{s}\right|^{\alpha} \\
&\left|I+s^{-1}\tilde{s}\right|^{-(\alpha+\beta)} e^{-\tr [\tilde{s}]} \left|I+\tilde{s}\tilde{s}\right|^{-\alpha} f\left(\tilde{a}(\tilde{s}^{-1}+\tilde{a})^{-1} \tilde{a}, \tilde{a}\right) .
\end{aligned}
\end{align}
Now, substitute the variable $\tilde{s}$ with
\begin{align*}
\tilde{r}:=\tilde{a}(\tilde{s}^{-1}+\tilde{a})^{-1}\tilde{a}=((\tilde{a}\tilde{s}\tilde{a})^{-1}+\tilde{a}^{-1})^{-1},
\end{align*}
or equivalently $\tilde{s}=\tilde{a}^{-1} (\tilde{r}^{-1}-\tilde{a}^{-1}) \tilde{a}^{-1}$.
Notice that $\tilde{s}\in\pos_d$ if and only if $\tilde{r}^{-1}-\tilde{a}^{-1}\in\pos_d$ if and only if $\tilde{a}-\tilde{r}\in\pos_d$.
Since the reference measure $\mu$ is $\GL_{d}$-invariant and also invariant under inversion (see Section~\ref{subsec:integration}), we have
\begin{align*}
|\tilde{a}\tilde{s}\tilde{a}|^{-(d+1)/2}\mu(\diff \tilde{s})=|\tilde{r}|^{-(d+1)/2}\mu(\diff \tilde{r}).
\end{align*}
Using these facts and carrying out a few algebraic manipulations, we obtain that~\eqref{eq:rightSide} agrees with~\eqref{eq:leftSide}, thus concluding the proof.
\end{proof}
 
\begin{corollary}\label{eigenequa}
We have the eigenfunction equation $\varphi = Q\varphi$.
\end{corollary}
\begin{proof}
By integrating the intertwining relation~\eqref{intertw2}, we have $K\,\Pi\, \mathbf{1} = Q\, K\, \mathbf{1}$, where $\mathbf{1}$ is the function that equals identically $1$ on $\pos_d^2$.
Noting that $\Pi \, \mathbf{1}=\mathbf{1}$ (as $\Pi$ is a Markov kernel) and $K\, \mathbf{1} = \varphi$ (by definition~\eqref{functionphi}), we arrive at the claim.
\end{proof}

Corollary~\eqref{eigenequa} ensures that $\overline{Q}$ is a Markov kernel.
The following normalised intertwining relation is an immediate consequence of~\eqref{eq:normalisedKernels} and Proposition~\ref{Int}.
\begin{corollary}\label{Maininter}
The following intertwining relation holds:
\begin{align}\label{intertwiningnormalised}
\overline{K}\,\Pi=\overline{Q}\,\overline{K},
\end{align}
where both sides are kernels from $\pos_d$ to $\pos_d^{2}$.
\end{corollary}

Let us now turn to the measure $\eta$ defined in~\eqref{kerinf2}, which, according to Theorem~\ref{R-P}, will be the initial distribution of the process $S$.
\begin{lemma}
\label{lemma:etaProbability}
The measure $\eta$ is a probability measure.
\end{lemma}
\begin{proof}
Observe that
\begin{align*}
\int_{\pos_d}\mu(\diff \tilde{s})\,\varphi(\tilde{s})\,|\tilde{s}|^{\alpha}e^{-\tr[\tilde{s}]}
&=\int_{\pos_d}\mu(\diff \tilde{s})\,|\tilde{s}|^{\alpha}e^{-\tr[\tilde{s}]}\int_{\pos_d}\mu(\diff a)\,|a|^{\alpha-\beta}\left|I+\tilde{s}a\right|^{-\alpha}e^{-\tr[a^{-1}]}\\
&=\int_{\pos_d}\mu(\diff x)\,|x|^{\beta-\alpha}|I+x^{-1}|^{-\alpha}\int_{\pos_d}\mu(\diff \tilde{s})\,|\tilde{s}|^{\beta}e^{-\tr[\tilde{s}(I+x)]}\\
&=\int_{\pos_d}\mu(\diff x)\,|x|^{\beta}|I+x|^{-(\alpha+\beta)}\int_{\pos_{d}}\mu(\diff y)\,|y|^{\beta}e^{-\tr[y]}\\
&=B_d(\alpha,\beta) \Gamma_{d}(\beta).
\end{align*}
For the above equalities, we have used~\eqref{functionphi}, the changes of variables $x:=\tilde{s}^{-1/2}a^{-1}\tilde{s}^{-1/2}$ (noting that $\mu(\diff x)=\mu(\diff a)$) and $y:=(I+x)^{1/2}\tilde{s}(I+x)^{1/2}$ (noting that $\mu(\diff y)=\mu(\diff \tilde{s})$), and, finally, the definitions of gamma and beta functions~\eqref{eq:gammaFn} and~\eqref{eq:betaFn}.
\end{proof}

\begin{lemma}\label{initialintertwining}
We have
\begin{align}\label{intertwiningnormalised}
\lambda\,\Pi=\eta\,\overline{K},
\end{align}
where both sides are probability measures on $\pos_d^{2}$.
\end{lemma}
\begin{proof}
This is checked directly by computing both sides of~\eqref{intertwiningnormalised}, similarly to the proof of Proposition~\ref{Int}.
\end{proof}

Finally, we are ready to prove Theorem~\ref{R-P}.

\begin{proof}[Proof of Theorem~\ref{R-P}]
Recall from~\eqref{eq:MarkovFunction} that $\phi(r,a):=(a-r)^{-1}ra^{-1}$ on the domain $D=\{(r,a)\in\pos_d^2 \colon a-r\in\pos_d\}$, so that $S(n)=\phi(R(n),A(n))$ for $n\geq 1$.

We would like to apply Theorem~\ref{MarkovFunctionsT}.
Notice first that we may view $\overline{K}$ as a Markov kernel from $\pos_d$ to $D$ and $\Pi$ as a Markov kernel from $D$ to $D$.
We may then view the intertwining relation $\overline{K}\,\Pi=\overline{Q}\,\overline{K}$ of Corollary~\ref{Maininter} as an equality of kernels from $\pos_d$ to $D$.
Notice now that
\begin{align*}
\phi^{-1}\{s\}&=\{(r,a)\in D: r=a(s^{-1}+a)^{-1}a\}\, , \quad s\in\pos_d\,.
\end{align*}
It is then clear from the definitions (see~\eqref{eq:normalisedKernels} and~\eqref{kernelK}) that the probability measure $\overline{K}(s;\cdot)$ is supported on $\phi^{-1}\{s\}$, namely $\overline{K}(s;\phi^{-1}\{s\})=1$, for all $s\in\pos_{d}$.

Since $(R(0),A(0))$ has distribution $\lambda$ and the transition kernel of $((R(n),A(n)), n\geq 1)$ is $\Pi$, the distribution of $(R(1),A(1))$ is $\lambda\Pi$, which equals $\eta\overline{K}$ by Lemma~\ref{initialintertwining}.
Theorem~\ref{MarkovFunctionsT} then implies that  $(\phi(R(n),A(n)), n\ge1)$ is a $\pos_d$-valued time-homogeneous Markov process, in its own filtration, with initial distribution $\eta$ and transition kernel $\overline{Q}$; moreover, \eqref{eq:R-P_conditional} holds.
\end{proof}

%Let us now specialise Theorem~\ref{R-P} to the case $s_0=\infty$.
%Let $X(0)$ be a random matrix in $\pos_d$, independent of $(X(n), n\geq1)$, with the inverse Wishart distribution $\invWish_{d}(\beta)$.
%Then, the pair $(X(0),X(0))$ in $\pos_d^{2}$ has clearly the distribution $\overline{K}(\infty; \cdot)$ defined by~\eqref{kerinf}.
%Let $\mathcal{R}=(\mathcal{R}(n), n\ge0)$ be a $\GL_{d}$-invariant random walk on $\pos_{d}$ with initial state $X(0)$ and increments $(X(n),n\ge1)$.
%Define the corresponding running sum in $\pos_{d}$ by
%\begin{align}\label{integralofR}
%\mathcal{A}(n):=\sum_{k=0}^{n}\mathcal{R}(k),\quad n\geq0.
%\end{align}
%The process $((\mathcal{R}(n), \mathcal{A}(n)),n\geq0)$ is a time-homogeneous Markov process in $\pos_{d}^{2}$ with transition kernel $\Pi$ and initial distribution $\overline{K}(\infty; \cdot)$.
%By Theorem~\ref{R-P}, we then have:
%
%\begin{theorem}\label{SIC}
%The process $\mathcal{S}=(\mathcal{S}(n),n\ge1)$ defined by 
%\begin{align}\label{secons}
%\mathcal{S}(n)&:=\mathcal{A}(n-1)^{-1}\mathcal{R}(n)\mathcal{A}(n)^{-1},\quad n\ge1,
%\end{align}
%is a time-homogeneous Markov process (in its own filtration) in $\pos_d$ with transition kernel $\overline{Q}$ and initial state $\mathcal{S}(1)$ distributed as $\overline{Q}(\infty;\cdot)$.
%Moreover, for any bounded measurable function $f:\pos_d^{2}\to\mathbb{R}$ and $n\ge1$, we have
%\begin{align}\label{R-P-1}
%\mathbb{E}\left[f(\mathcal{R}(n),\mathcal{A}(n)) \mid \mathcal{S}(1),\dots,\mathcal{S}(n-1),\mathcal{S}(n)\right]=\overline{K}\,f (\mathcal{S}(n)),\quad\text{a.s}.
%\end{align}
%\end{theorem}

\subsection{Dufresne type identity}
\label{subsec:Dufresne}

Here we prove a discrete, matrix analogue of the classical Dufresne identity: under the additional assumption that $\beta-\alpha>\frac{d-1}{2}$, the running sum $A(n)$ in~\eqref{runningsum} converges a.s.\ as $n\to\infty$ and the limit has an inverse Wishart distribution.
We deduce this from Theorem~\ref{R-P}, using also a result on Lyapunov exponents of random walks in $\pos_d$ with Beta type II distributed increments from Appendix~\ref{app:lyapunovExponents}.

\begin{proposition}\label{convergenceA}
If $\beta-\alpha>\frac{d-1}{2}$, the series
\begin{align}\label{Ainfty}
A_{\infty}
:=\lim_{n\to\infty}A(n)
=\sum_{k=0}^{\infty}R(k)
\end{align}
converges a.s.\ in $\pos_d$.
\end{proposition}

\begin{proof}
By Corollary~\ref{Lyapunov}, we have
\begin{align*}
\gamma:=\lim_{n\to\infty}\frac{1}{n}\log \lambda_{\max}(R(n))=\psi\left(\alpha\right)-\psi\left(\beta-\frac{d-1}{2}\right),\quad a.s.,
\end{align*}
where $\psi$ is the digamma function defined in~\eqref{digamma}.
Since $\psi$ is strictly increasing on $(0,\infty)$ and, by assumption, $\beta-\alpha>\frac{d-1}{2}$, we have $\gamma<0$ almost surely.
As a consequence, $\lambda_{\max}(R(n))$ vanishes exponentially fast, as $n\to\infty$, a.s., hence the series $\sum_{k=0}^{\infty} \lambda_{\max}(R(k))$ converges a.s.
The space of $d\times d$ real symmetric matrices with the norm $\lVert \cdot \rVert:= \lvert \lambda_{\max}(\cdot) \rvert$ is complete; therefore, the series $A_{\infty}$ defined in~\eqref{Ainfty} converges absolutely a.s.\ in such a space (with respect to $\lVert \cdot \rVert$ or any other norm).
That $A_{\infty}$ actually takes values in $\pos_d$ follows from the fact that, by~\eqref{lambdasum}, $\lambda_{\min}(A_{\infty}) \geq \lambda_{\min}(A(0)) >0$, as $A(0)$ takes values in $\pos_d$.
\end{proof}

\begin{theorem}\label{Dufresne}
If $\beta-\alpha>\frac{d-1}{2}$, then $A_{\infty}$ has the inverse Wishart distribution $\invWish_{d}(\beta-\alpha)$.
\end{theorem}

\begin{proof}
Using~\eqref{functional} and~\eqref{lambdasum}-\eqref{lambdaproduct}, we have, a.s.,
\begin{align*}
\begin{split}
\lambda_{\max}(S(n))
&\leq \lambda_{\max}\left(A(n-1)^{-1}\right)
\lambda_{\max}(R(n))
\lambda_{\max}\left(A(n)^{-1}\right) \\
&= \lambda_{\min}(A(n-1))^{-1}
\lambda_{\max}(R(n))
\lambda_{\min}(A(n))^{-1} \\
&\leq \lambda_{\min}(A(0))^{-2} 
\lambda_{\max}(R(n)) .
\end{split}
\end{align*}
Since $A(0)$ is a random matrix in $\pos_d$, $\lambda_{\min}(A(0))^{-2}$ is an a.s.\ real positive random variable.
On the other hand, we showed in the proof of Proposition~\ref{convergenceA} that $\lambda_{\max}(R(n))\to0$ as $n\to\infty$ a.s., hence $\lambda_{\max}(S(n))\to0$ a.s., i.e.\ $S(n)\to0$ a.s.

Let $\nu_{n}$ be the distribution of $S(n)$, for $n\ge1$, so that, by the above argument, $\nu_{n}$ converges weakly to the Dirac measure at $0$.
Let $g$ be any bounded and continuous function on $\pos_{d}$.
By averaging~\eqref{eq:R-P_conditional} with $f(x,y):=g(y)$, we obtain
\begin{align*}
\mathbb{E}\left[g(A(n))\right]&=\int_{\pos_d}\nu_{n}(\diff s)\,\overline{K}\,f (s)=\int_{\pos_d}\nu_{n}(\diff s)\,K^{\dag}g (s),
\end{align*}
where
\begin{align*}
K^{\dag}(s; \diff a):=\frac{1}{\varphi(s)}k(s;\diff a)=\frac{1}{\varphi(s)}\,|a|^{\alpha-\beta}\left|I+sa\right|^{-\alpha}e^{-\tr[a^{-1}]}\mu(\diff a),
\end{align*}
and $\varphi(s)$ is the normalisation given in~\eqref{functionphi}.
The function $s\mapsto K^{\dag}g(s)$ is bounded on $\pos_{d}$.
Moreover, note that $K^{\dag}(s; \diff a)$ converges weakly, as $s\to0$, to the inverse Wishart distribution with parameter $\beta-\alpha$:
\begin{align}\label{Kzero}
K^{\dag}(0; \diff a):=\Gamma_{d}(\beta-\alpha)^{-1}|a|^{-(\beta-\alpha)}e^{-\tr[a^{-1}]}\mu(\diff a).
\end{align}
It follows that $s\mapsto K^{\dag}g(s)$ is continuous at $s=0$.
Since $\nu_{n}$ converges weakly to the Dirac measure at zero, we have $\mathbb{E}\left[g(A(n))\right]\to K^{\dag}g(0)$, as $n\to\infty$, as desired.
\end{proof}

\section{Stochastic matrix equations}
\label{sec:stochasticEqns}

In this section, we give a completely different proof of the Dufresne type identity of Section~\ref{subsec:Dufresne}, linking it to the solution to certain stochastic matrix equations.

Recall that $T^w$ is the symmetrised multiplication operation~\eqref{eq:mult}.
Throughout this section, $w$ will be an arbitrary function satisfying~\eqref{eq:splitFn}; thus, for the sake of notational simplicity, we will drop the superscript $w$ in $T^w$.
We will also denote by $\mathscr{L}(X)$ the distribution of a random matrix $X$ in $\pos_d$.

\subsection{Matrix Kesten recursions}
\label{subsec:stochasticRecursions}

%Equation~\eqref{stoequation} has an interesting interpretation in terms of stationary distributions for certain matrix recursions in $\pos_d$.
Let $x\in\pos_d$ and consider a family $(X(n),n\ge1)$ of i.i.d.\ $\ort_d$-invariant random matrices in $\pos_d$.
Define a Markov process $\xi=(\xi(n), n\ge0)$ in $\pos_d$ by the recursion
\begin{align}
\label{srec}
\xi(0):=x, \qquad\qquad
\xi(n):=T_{X(n)}(I+\xi(n-1)) \qquad \text{for } n\ge1.
\end{align}
When $d=1$, this process is sometimes referred to as \emph{Kesten recursion} and has been widely considered in the mathematical literature (see e.g.~\cite{Kes73, Ver, CL, Gol, D, CL2, DF}).
Note that the distribution $\mathscr{L}(Z)$ of a random matrix $Z$ in $\pos_d$ is stationary for the Markov process $\xi$ if and only if
\begin{align}\label{stoequationgeneral}
Z\stackrel{\mathrm{d}}{=}T_{X}(I+Z),
\end{align}
where $X\stackrel{\mathrm{d}}{=}X(1)$ is independent of $Z$.
Define now $F_n(z) := T_{X(n)}(I+z)$ for any $z\in\pos_d$ and $n\geq 1$.
By iterating~\eqref{srec}, one may write
\begin{equation}
\label{equalindistribution}
\xi(n)
= F_n \circ \cdots \circ F_1(x)
\stackrel{\mathrm{d}}{=}
F_1 \circ \cdots \circ F_n(x),
\end{equation}
for all $n\geq 1$, where the distributional equality is due to the fact that $(X(n), n\geq1)$ are i.i.d.

As an application of a criterion appeared in~\cite{Letac, CL}, we obtain the following.

\begin{proposition}\label{finiteZ}
Assume that the series
\begin{align}\label{seriesZ}
Z_{\infty}:=\sum_{k=1}^{\infty}T_{X(1)}\circ T_{X(2)}\circ\cdots \circ T_{X(k)}(I)
\end{align}
converges a.s.
Then, $\mathscr{L}(Z_{\infty})$ is the unique stationary distribution for the Markov process $\xi$, i.e.\ $\mathscr{L}(Z_{\infty})$ is the unique solution to~\eqref{stoequationgeneral}.
Furthermore, $\mathscr{L}(Z_{\infty})$ is $\ort_d$-invariant.
\end{proposition}
\begin{proof}
By~\cite[Proposition 1]{CL}, if $F_1 \circ \cdots \circ F_n(x)$ converges a.s.\ to a limit $Z$ that does not depend on $x$, then $\xi(n)
= F_n \circ \cdots \circ F_1(x)$ has a unique stationary distribution given by $\mathscr{L}(Z)$.
To prove the first claim, it then suffices to show that $F_1 \circ \cdots \circ F_n(x)$ converges a.s.\ to $Z_{\infty}$.
Notice that
\begin{equation}
\label{eq:inverseComposition}
F_1 \circ \cdots \circ F_n(x)
=T_{X(1)}\circ T_{X(2)}\circ\cdots \circ T_{X(n)}(x)+\sum_{k=1}^{n}T_{X(1)}\circ T_{X(2)}\circ\cdots \circ T_{X(k)}(I) .
\end{equation}
By hypothesis, the sum in~\eqref{eq:inverseComposition} converges a.s.\ as $n\to\infty$, hence $T_{X(1)}\circ T_{X(2)}\circ\cdots \circ T_{X(n)}(I)$ converges to the zero matrix a.s.
By~\eqref{lambdasum} and the fact that the matrices $ab$ and $ba$ have the same eigenvalues (for $a,b\in\GL_d$), we have
\[
\lambda_{\max}(T_{X(1)}\circ T_{X(2)}\circ\cdots \circ T_{X(n)}(x))
\leq \lambda_{\max}(x) \, \lambda_{\max}(T_{X(1)}\circ T_{X(2)}\circ\cdots \circ T_{X(n)}(I)) .
\]
The right-hand side of the latter inequality vanishes a.s., hence $T_{X(1)}\circ T_{X(2)}\circ\cdots \circ T_{X(n)}(x)$ converges to the zero matrix a.s., for all $x$.
We conclude that~\eqref{eq:inverseComposition} converges a.s.\ to $Z_{\infty}$, as desired.

It remains to check that $\mathscr{L}(Z_{\infty})$ is $\ort_d$-invariant.
Notice first that, by~\eqref{equalindistribution}, \eqref{seriesZ} and~\eqref{eq:inverseComposition}, $\mathscr{L}(Z_{\infty})$ is the $n\to\infty$ limit in distribution of $\xi(n)$, when the initial state of $\xi$ is taken to be $x=0$.
On the other hand, when $x=0$, it can be checked inductively, using Proposition~\ref{prop:T}-\eqref{it:T_orthInv}, that $\xi(n)$ is $\ort_d$-invariant for all $n\geq 1$.
Therefore, $\mathscr{L}(Z_{\infty})$ is the weak limit of a sequence of $\ort_d$-invariant distributions.
\end{proof}

Instead of~\eqref{srec}, consider now the recursion $\xi'=(\xi'(n), n\ge0)$ defined by
\begin{align}
\label{srecGL}
\xi'(0):=x, \qquad\qquad
\xi'(n):=T_{I+\xi'(n-1)}(X(n)) \qquad \text{for } n\geq 1.
\end{align}
This process was studied, under various aspects, in~\cite{GBLD} and called \emph{matrix Kesten recursion} in that article.

\begin{corollary}\label{coro:invariantBothProcessesGen}
Under the assumption that the series $Z_{\infty}$ defined in~\eqref{seriesZ} converges a.s., $\mathscr{L}(Z_{\infty})$ is the unique $\ort_d$-invariant stationary distribution for both processes $\xi$ and $\xi'$.
\end{corollary}
\begin{proof}
Notice that $T_Y(X) \stackrel{\mathrm{d}}{=} T_X(Y)$ if $X$ and $Y$ are both $\ort_d$-invariant (cfr.\ Lemma~\ref{convolution}).
Therefore, the $\ort_d$-invariant stationary distributions for $\xi$ and $\xi'$ coincide.
The claim then follows from Proposition~\ref{finiteZ}.
\end{proof}

\subsection{Alternative proof of the Dufresne type identity}

We now give an alternative proof of Theorem~\ref{Dufresne}, by explicitly solving a stochastic matrix equation of type~\eqref{stoequationgeneral}.
To do so, we use various classical properties that relate Wishart and matrix Beta distributions, all collected in Appendix~\ref{app:WishartBeta} for the reader's convenience.

For this alternative proof, it will be convenient to work with a $\GL_d$-invariant random walk $R$ realised by the second construction~\eqref{eq:RWmod} given in Proposition~\ref{prop:invRW}:
\[
R(n)
:= T_{R(0)}\circ T_{X(1)}\circ\dots\circ T_{X(n)}(I), \qquad n\ge0 .
\]
As in Section~\ref{subsec:Dufresne}, we fix $\alpha>\frac{d-1}{2}$ and $\beta-\alpha>\frac{d-1}{2}$, and take $R(0), X(1), X(2),\dots$ to be independent random matrices such that
\begin{itemize}
	\item $R(0)$ has the inverse Wishart distribution $\invWish_{d}(\beta)$;
	\item $X(n)$ has the Beta type II distribution $\BetaII_{d}(\alpha,\beta)$, for all $n\ge1$.
\end{itemize}

By Proposition~\ref{convergenceA}, the series
\begin{align}\label{Ainfty2}
A_\infty:=\sum_{k=0}^{\infty}T_{R(0)}\circ T_{X(1)}\circ\dots\circ T_{X(k)}(I)
\end{align}
converges almost surely.
Note that we may write
\begin{align}\label{Ainfty3}
A_{\infty}=T_{R(0)}(I+Z_{\infty}),
\end{align}
where
\begin{align}\label{Zinfty}
Z_{\infty}:=\sum_{k=1}^{\infty}T_{X(1)}\circ T_{X(2)}\circ\dots\circ T_{X(k)}(I)
\end{align}
is a random matrix in $\pos_d$, independent of $R(0)$.
We will now show that $\mathscr{L}(Z_{\infty})$ is the unique solution to a stochastic equation of type~\eqref{stoequationgeneral}.

\begin{proposition}\label{SAE}
Let $X \sim \BetaII_{d}(\alpha,\beta)$.
Consider the stochastic equation
\begin{align}\label{stoequation}
Z\stackrel{{\rm d}}{=}T_{X}(I+Z),
\end{align}
for a random matrix $Z$ in $\pos_d$ independent of $X$.
Then, $\mathscr{L}(Z_{\infty})$ is the unique solution to equation~\eqref{stoequation}, and $\mathscr{L}(Z_{\infty})=\BetaII_{d}(\alpha,\beta-\alpha)$.
\end{proposition}

\begin{rem}
Proposition~\ref{SAE} is a generalisation of the $d=1$ case addressed by Chamayou and Letac~\cite[Example 9]{CL} and Dufresne~\cite[Remark 1]{D}. 
\end{rem}

\begin{proof}[Proof of Proposition~\ref{SAE}]
As the series $A_\infty$ converges a.s.\ by Proposition~\ref{convergenceA}, the series $Z_{\infty}$ also does.
Then, by Proposition~\ref{finiteZ}, the distribution $\mathscr{L}(Z_{\infty})$ is the unique solution to equation~\eqref{stoequation}.

It remains to prove that~\eqref{stoequation} is satisfied for $Z\sim \BetaII_{d}(\alpha,\beta-\alpha)$.
By Proposition~\ref{relationswishart}-\eqref{it:relationswishart_betaI-II}, $Z\sim \BetaII_{d}(\alpha,\beta-\alpha)$ if and only if $I+Z\sim \invBetaI_d(\beta-\alpha,\alpha)$.
Therefore, using the notation $\odot$ from Appendix~\ref{app:WishartBeta}, our aim is to show that
\[
\BetaII_d(\alpha,\beta-\alpha) = \BetaII_d(\alpha,\beta) \odot \invBetaI_d(\beta-\alpha,\alpha).
\]
Using Proposition~\ref{relationswishart}-\eqref{it:relationswishart_betaII}, Lemma~\ref{convolution} and the identity~\eqref{invW}, we obtain
\[
\begin{split}
&\BetaII_d(\alpha,\beta) \odot \invBetaI_d(\beta-\alpha,\alpha)
=\left(\invWish_d(\beta) \odot \Wish_d(\alpha)\right) \odot \invBetaI_d(\beta-\alpha,\alpha) \\
=\, &\left(\invWish_{d}(\beta)\odot \invBetaI_{d}(\beta-\alpha,\alpha)\right)\odot \Wish_{d}(\alpha)
=\invWish_{d}(\beta-\alpha)\odot \Wish_{d}(\alpha)
=\BetaII_{d}(\alpha,\beta-\alpha),
\end{split}
\]
as desired.
\end{proof}

\begin{corollary}\label{Remark1}
$A_{\infty}$ has the inverse Wishart distribution $\invWish_{d}(\beta-\alpha)$.
\end{corollary}

\begin{proof}
We have $Z_{\infty}\sim\BetaII_{d}(\alpha,\beta-\alpha)$ by Proposition~\ref{SAE}, hence $I+Z_\infty\sim \invBetaI_{d}(\beta-\alpha,\alpha)$ by Proposition~\ref{relationswishart}-\eqref{it:relationswishart_betaI-II}.
It then follows from~\eqref{Ainfty3} and~\eqref{invW} that
\begin{align*}
A_{\infty}
=T_{R(0)}(I+Z_{\infty})
\sim \invWish_{d}(\beta)\odot \invBetaI_{d}(\beta-\alpha,\alpha)
=\invWish_{d}(\beta-\alpha),
\end{align*}
as desired.
\end{proof}

Finally, combining Proposition~\ref{SAE} with Corollary~\ref{coro:invariantBothProcessesGen}, we obtain the $\ort_d$-invariant stationary distributions of the matrix Kesten recursions of Section~\ref{subsec:stochasticRecursions} in the case of Beta type II increments.

\begin{corollary}\label{coro:invariantBothProcesses}
The Beta type II distribution $\BetaII_{d}(\alpha,\beta-\alpha)$ is the unique $\ort_d$-invariant stationary distribution for both processes $\xi$ and $\xi'$, defined in~\eqref{srec} and~\eqref{srecGL}, respectively, when the i.i.d.\ increments $(X(n),n\ge1)$ have the Beta type II distribution $\BetaII_{d}(\alpha,\beta)$.
\end{corollary}

\vspace{2mm}

\subsection*{Acknowledgements}
The authors thank the anonymous referees for their helpful comments and suggestions, which have led to a much improved version of the paper.

\appendix

\section{Random walk constructions}
\label{app:InvRW}

The main goal of this appendix is to prove Proposition~\ref{prop:invRW}.
Let us start with a few preliminary lemmas.

\begin{lemma}
\label{lemma:ortinv}
Let $X$ be an $\ort_{d}$-invariant random matrix in $\pos_d$.
Let $Z$ be any random variable independent of $X$.
Let $O$ be a random matrix in $\ort_d$ that is a deterministic function of $Z$.
Then, $(O^{\top} X O, Z) \stackrel{\mathrm{d}}{=} (X,Z)$.
In particular, $O^{\top}XO$ is independent of $Z$ and $O$.
\end{lemma}
\begin{proof}
Conditioned on $Z$, the random matrix $O$ is constant, as a deterministic function of $Z$.
Since $X$ is $\ort_d$-invariant and independent of $Z$, we then have
\begin{align*}
\begin{split}
\mathbb{E}\left[f(O^{\top}XO) g(Z) \right]
&=\mathbb{E}\left[ \mathbb{E} \left[  f(O^{\top} X O) g(Z) \bigm\vert Z\right]\right] \\
&=\mathbb{E}\left[ g(Z)\mathbb{E} \left[  f(X) \bigm\vert Z\right]\right]
=\mathbb{E}\left[ f(X) \right]
\mathbb{E}\left[ g(Z) \right]
\end{split}
\end{align*}
for any bounded and measurable functions $f,g$.
This proves that $(O^{\top} X O, Z) \stackrel{\mathrm{d}}{=} (X,Z)$ and that $O^{\top} X O$ is independent of $Z$.
Since $O$ is a deterministic function of $Z$, $O^{\top} X O$ is also independent of $O$.
\end{proof}

\begin{lemma}
\label{lemma:prodRuleDoesNotMatter0}
Let $A$ and $B$ be random matrices in $\GL_d$ with $A^{\top}A = B^{\top}B$.
Let $X$ be an $\ort_d$-invariant random matrix in $\pos_d$, independent of $(A,B)$.
Then, $(A^{\top} X A, A, B) \stackrel{\mathrm{d}}{=} (B^{\top} X B, A, B)$.
\end{lemma}
\begin{proof}
The matrix $O:= AB^{-1}$ is orthogonal, since, by hypothesis,
\[
O^{\top} O
= (B^{\top})^{-1} A^{\top} A B^{-1}
= (B^{\top})^{-1} B^{\top} B B^{-1}
= I.
\]
Moreover, $O$ is a deterministic function of $Z:=(A,B)$, which, by hypothesis, is independent of $X$.
We deduce from Lemma~\ref{lemma:ortinv} that $(O^{\top} X O, A, B) \stackrel{\mathrm{d}}{=} (X,A,B)$.
Applying the measurable function $f(x,a,b):= (b^{\top}xb,a,b)$, for $x,a,b\in\GL_d$, to this distributional equality, we obtain
\[
(A^{\top} X A, A, B)
= f(O^{\top} X O, A, B)
\stackrel{\mathrm{d}}{=} f(X,A,B)
= (B^{\top} X B, A, B) ,
\]
as desired.
\end{proof}

Recall from Section~\ref{subsec:InvRW} that $w$ denotes any measurable function $w\colon \pos_d\to\GL_d$ such that $x=w(x)^{\top} w(x)$ for all $x\in\pos_d$.
Given $y\in\pos_d$, recall also the `multiplication operation' by $y$ defined in~\eqref{eq:mult}, denoted by $\mult_{y}\colon \pos_d \to \pos_d$.

\begin{proposition}
\label{prop:T}
Let $X$ and $Y$ be independent random matrices in $\pos_d$, and assume that $X$ is $\ort_d$-invariant.
Then:
\begin{enumerate}[(i)]
\item \label{it:T_prodRuleDoesNotMatter} The distribution of $T^{w}_Y(X)$ is the same for any choice of the function $w$ satisfying~\eqref{eq:splitFn}.
\item
\label{it:T_inverse} The distributions of $[T^{w}_Y(X)]^{-1}$ and $T^{w}_{Y^{-1}}(X^{-1})$ are equal.
\item
\label{it:T_orthInv} If, in addition, $Y$ is $\ort_d$-invariant, then $T^{w}_Y(X)$ also is.
\end{enumerate}
\end{proposition}

\begin{proof}
\begin{enumerate}[(i)]
\item Let $w,w'$ be any two functions satisfying~\eqref{eq:splitFn}.
Letting $A:= w(Y)$ and $B:= w'(Y)$, we have that $X$ is independent of $(A,B)$ and, by~\eqref{eq:splitFn}, $A^{\top}A = Y = B^{\top} B$.
Therefore, by Lemma~\ref{lemma:prodRuleDoesNotMatter0}, we have $A^{\top} X A \stackrel{\mathrm{d}}{=} B^{\top} X B$, i.e.\ $T^{w}_Y(X) \stackrel{\mathrm{d}}{=} T^{w'}_Y(X)$.
\item If $w$ is the square root function, then 
\[
[T^{w}_Y(X)]^{-1}
= \big[Y^{1/2} X Y^{1/2}\big]^{-1}
= (Y^{-1})^{1/2} X^{-1} (Y^{-1})^{1/2}
= T^{w}_{Y^{-1}}(X^{-1}).
\]
Noticing that $X^{-1}$ is $\ort_d$-invariant (since $X$ is), the claim for general $w$ follows from~\eqref{it:T_prodRuleDoesNotMatter}.
\item By~\eqref{it:T_prodRuleDoesNotMatter}, without loss of generality we may take $w$ to be the square root function.
For any $k\in\ort_d$, we have
\[
\begin{split}
k^{\top} T^w_Y(X) k
&= k^{\top} Y^{1/2} kk^{\top} X kk^{\top} Y^{1/2} k \\
&= (k^{\top} Y k)^{1/2} (k^{\top} X k) (k^{\top} Y k)^{1/2}
= T^w_{k^{\top} Y k}(k^{\top} X k)
\stackrel{\mathrm{d}}{=} T^w_Y(X).
\end{split}
\]
Here, the last (distributional) equality follows from the fact that, as $X$ and $Y$ are independent and both $\ort_d$-invariant, $(k^{\top} X k, k^{\top} Y k)$ equals in distribution $(X,Y)$.
\qedhere
\end{enumerate}
\end{proof}

\begin{rem}
Instead of the operator $\mult_y$ defined in~\eqref{eq:mult}, consider the alternative operator
\begin{equation}
\label{eq:multmod}
\multmod_y\colon \pos_d \to \pos_d ,
\qquad
\multmod_y(x) := w(y) \, x\, w(y)^{\top}, \quad x\in\pos_{d}.
\end{equation}
Clearly, $\mult_y = \multmod_y$ when $w$ is the square root function.
However, Proposition~\ref{prop:T} is, for a general $w$, no longer true when using $\multmod$ instead of $\mult$.
For example, let $X$ be a Wishart matrix and let $Y$ be an inverse Wishart matrix, independent of $X$.
Then, when $w$ is the square root function, the distribution of $\multmod_{Y}(X) = \mult_{Y}(X)$ is Beta type II (see Proposition~\ref{relationswishart}-\eqref{it:relationswishart_betaII}).
On the other hand, when $w$ is the Cholesky function, $\multmod_{Y}(X)$ has an explicit distribution (see e.g.~\cite[Theorem~5.4.2]{GN}), which is not Beta II, nor is it $\ort_d$-invariant.
\end{rem}

\begin{proof}[Proof of Proposition~\ref{prop:invRW}]
For any $w$, $\RW$ is clearly a Markov process with initial state $M$.
Furthermore, by Proposition~\ref{prop:T}-\eqref{it:T_prodRuleDoesNotMatter}, the time-$n$ transition kernel does not depend on the choice of the function $w$, for any $n$.
As a result, the law of the process $\RW$ does not depend on the choice of $w$ either.

We now show that, for any fixed function $w$ satisfying~\eqref{eq:splitFn}, $\RWmod$ is a $\GL_d$-invariant Markov process with the same law as $\RW$.
Fix $a\in \GL_d$ and, for any $n\geq 0$, let 
\[
A(n):= w(X(n)) \cdots w(X(1)) w(M) a,
\qquad
B(n):= w(a^{\top} \RWmod(n) a) .
\]
We then have
\[
A(n)^{\top} A(n)
= a^{\top} \RWmod(n) a
= B(n)^{\top} B(n) .
\]
Since $X(n)$ is independent of $(A(n-1), B(n-1))$, it follows from Lemma~\ref{lemma:prodRuleDoesNotMatter0} that
\[
(A(n-1)^{\top} X(n) A(n-1), B(n-1))
\stackrel{\mathrm{d}}{=}
(B(n-1)^{\top} X(n) B(n-1), B(n-1))
\]
for all $n\geq 1$.
Now, we have
\begin{align*}
A(n-1)^{\top} X(n) A(n-1)
&= A(n)^{\top} A(n)
= a^{\top} \RWmod(n) a , \\
B(n-1)^{\top} X(n) B(n-1)
&= T^w_{a^{\top}\RWmod(n-1)a}(X(n)) .
\end{align*}
Since $a^{\top} \RWmod(n-1) a$ is a deterministic function of $B(n-1)$, we then have
\begin{equation}
\label{eq:equalityInLawRWs}
(a^{\top} \RWmod(n) a, a^{\top} \RWmod(n-1) a)
\stackrel{\mathrm{d}}{=}
(T^w_{a^{\top}\RWmod(n-1)a}(X(n)), a^{\top} \RWmod(n-1) a) .
\end{equation}
It follows that $(a^{\top}\RWmod(n)a, n\geq 0)$ is a Markov process whose time-$n$ transition kernel does not depend on the choice of $a\in\GL_d$, for all $n$: in other words, $\RWmod$ is $\GL_d$-invariant.
Letting $a=I$ in~\eqref{eq:equalityInLawRWs}, we also see that $\RWmod$ has the same time-$n$ transition kernel as $\RW$, for all $n$; since $\RWmod$ and $\RW$ also have the same initial state $\RWmod(0)=M=\RW(0)$, we conclude that they have the same law.
\end{proof}

\section{Wishart and matrix Beta distributions}
\label{app:WishartBeta}

In this appendix we collect some important properties that relate Wishart and matrix Beta distributions (defined in Section~\ref{subsec:orthInvariant}) and present their Cholesky decompositions.

By Proposition~\ref{prop:T}, if $Y\sim\eta$ and $X\sim\nu$ are two independent random matrices in $\pos_{d}$ with $\ort_{d}$-invariant distributions, then the distribution of $\mult_{Y}(X)$ is the same for any choice of $w$ and is itself $\ort_{d}$-invariant; throughout, we denote such a distribution by $\eta\odot\nu$.
In fact, the operation $\odot$ on the space of $\ort_d$-invariant distributions may be expressed in terms of the usual convolution with respect to the Haar measure on $\GL_{d}$ (see e.g.~\cite[Lemma 1.1.1]{T}).
This, in particular, implies:
 
\begin{lemma}\label{convolution}
The operation $\odot$ is commutative and associative.
\end{lemma}

Recall also the alternative operator $\multmod_{y}\colon \pos_d \to \pos_d$, for $y\in\pos_d$, defined in~\eqref{eq:multmod}.
The next proposition is a collection of known results on matrix variate distributions, see e.g.~\cite[Chapter 5]{GN} and~\cite[Chapter 5]{Ma}.
For item~\eqref{it:relationswishart_betaI}, see also~\cite[Theorem 3.1-(i)]{CaL}; item~\eqref{it:relationswishart_betaII} follows from~\cite[Theorem 5.2.5]{GN} together with Proposition~\ref{prop:T}-\eqref{it:T_prodRuleDoesNotMatter}.

\begin{proposition}\label{relationswishart}
Let $\alpha>\frac{d-1}{2}$ and $\beta>\frac{d-1}{2}$.
Consider independent Wishart matrices $Y_{\alpha}\sim \Wish_{d}(\alpha)$ and $Y_{\beta}\sim \Wish_{d}(\beta)$.
Then:
\begin{enumerate}[(i)]
	\item \label{it:relationswishart_sum} $Y_{\alpha}+Y_{\beta}\sim \Wish_{d}(\alpha+\beta)$.
	\item \label{it:relationswishart_betaI} $\multmod_{(Y_{\alpha}+Y_{\beta})^{-1}}(Y_{\alpha})\sim \BetaI_{d}(\alpha,\beta)$ for any choice of $w$.
	\item \label{it:relationswishart_betaII} $\mult_{Y_{\beta}^{-1}}(Y_{\alpha})\sim \BetaII_{d}(\alpha,\beta)$ for any choice of $w$, i.e.\ $\BetaII_{d}(\alpha,\beta)=\invWish_{d}(\beta)\odot \Wish_{d}(\alpha)$.
	\item \label{it:relationswishart_betaIIinv} If $Z\sim \BetaII_{d}(\alpha,\beta)$, then $Z^{-1}\sim \BetaII_{d}(\beta,\alpha)$.
	\item \label{it:relationswishart_betaI-II} If $U\sim\BetaI_{d}(\alpha,\beta)$, then $V:=U^{-1}-I\sim\BetaII_{d}(\beta,\alpha)$.
	Conversely, if $V\sim\BetaII_{d}(\beta,\alpha)$, then $U:=(I+V)^{-1}\sim\BetaI_{d}(\alpha,\beta)$.
\end{enumerate}
\end{proposition}

Notice that~\eqref{it:relationswishart_betaI} and~\eqref{it:relationswishart_betaII} no longer hold if replacing the $\multmod$-operator with $\mult$ or vice versa (unless $w$ is the square root function, for which the two operators coincide).

\subsection{Matrix Lukacs theorem}
\label{app:matrixLukacs}

A well-known characterisation of the gamma distribution, due to Lukacs~\cite{L}, states that if $X$ and $Y$ are two positive, non-Dirac, and independent random variables, then $\frac{X}{X+Y}$ and $X+Y$ are independent if and only if $X$ and $Y$ have gamma distributions with the same scale parameter.
In the matrix setting, there is a generalisation in terms of the Wishart distribution, due to Olkin and Rubin~\cite[Theorems 1 and 2]{OR} and Casalis and Letac~\cite[Theorem 1.2]{CaL}.
We state below only the part of such a generalisation that is necessary for our purposes.

\begin{theorem}[Matrix Lukacs theorem]\label{Lukacs's identity}
If $X$ and $Y$ are two independent Wishart matrices (with possibly different parameters), then $\multmod_{(X+Y)^{-1}}(X)$ and $X+Y$ are independent, for any choice of $w$.
\end{theorem}

\begin{rem}
As observed in~\cite{OR}, this result is no longer valid, in general, if we replace $\multmod_{(X+Y)^{-1}}(X)$ with $\multmod_{X}((X+Y)^{-1})$ or if we replace the $\multmod$-operator with $\mult$.
\end{rem}

\begin{corollary}
\label{coro:wishartDecomposition}
For $\alpha>\frac{d-1}{2}$ and $\beta>\frac{d-1}{2}$, we have
\begin{align}\label{Wis}
\Wish_{d}(\alpha)&=\Wish_{d}(\alpha+\beta)\odot \BetaI_{d}(\alpha,\beta), \\
\label{invW}
\invWish_{d}(\alpha)&=\invWish_{d}(\alpha+\beta)\odot \invBetaI_{d}(\alpha,\beta).
\end{align}
\end{corollary}
\begin{proof}
For simplicity, and without loss of generality, we may choose the square root function $w(x):= x^{1/2}$ in the definition of $\odot$.
In this case, we have $\mult_y = \multmod_y$ and $(\mult_y)^{-1} = \mult_{y^{-1}}$, for all $y\in\pos_d$.
Let $X\sim \Wish_{d}(\alpha)$ and $Y \sim \Wish_{d}(\beta)$ be independent.
We then have $X=\mult_{X+Y}(\mult_{(X+Y)^{-1}}(X))$.
By Proposition~\ref{relationswishart}, items~\eqref{it:relationswishart_sum} and~\eqref{it:relationswishart_betaI}, we have $X+Y\sim \Wish_d(\alpha+\beta)$ and $\mult_{(X+Y)^{-1}}(X) \sim \BetaI_{d}(\alpha, \beta)$.
On the other hand, by Theorem~\ref{Lukacs's identity}, $\mult_{(X+Y)^{-1}}(X)$ is independent of $X+Y$.
Therefore, we have $\mult_{X+Y}(\mult_{(X+Y)^{-1}}(X)) \sim \Wish_{d}(\alpha+\beta)\odot \BetaI_{d}(\alpha,\beta)$, which proves~\eqref{Wis}.
Identity~\eqref{invW} follows from~\eqref{Wis} together with Proposition~\ref{prop:T}-\eqref{it:T_inverse}.
\end{proof}

\begin{rem}
It is instructive to prove Corollary~\ref{coro:wishartDecomposition} by using a general function $w$ satisfying~\eqref{eq:splitFn} that is not the square root.
One may then better distinguish the different roles that the two operators $\mult$ and $\multmod$ play and their properties (with particular reference to Proposition~\ref{relationswishart} and Theorem~\ref{Lukacs's identity}).
\end{rem}

\subsection{Cholesky decompositions}
\label{subsec:Bartlett}

The Cholesky decomposition of a Wishart matrix is standard and widely known as \emph{Bartlett decomposition}.
We now give a quick proof of this and deduce the Cholesky decompositions of inverse Wishart and Beta type II random matrices.

Let $\alpha$ and $c_1,\dots,c_d$ be real parameters such that $\alpha-\frac{c_k-1}{2}>0$ for all $1\leq k\leq d$.
For the sake of convenience, let us define the distribution $\mathfrak{B}_d(\alpha;c_1,\dots,c_d)$ as the distribution of an upper triangular matrix $U$ with independent entries such that
\begin{itemize}
\item $U_{k,k}^2\sim \Wish_1\left(\alpha-\frac{c_k-1}{2}\right)$, for $1\leq k\leq d$;
\item $U_{i,j}\sim \mathcal{N}(0,1/2)$, for $1\leq i< j\leq d$.
\end{itemize}
Defining
\begin{equation}
\label{eq:revPermutation}
\omega=(\delta_{i,d+1-j})_{1\leq i,j\leq d}
\end{equation}
to be the matrix corresponding to the reverse permutation, it is immediate to check that
\begin{equation}
\label{eq:BartlettEquiv}
U\sim \mathfrak{B}(\alpha; c_1,\dots,c_d)
\qquad \text{if and only if} \qquad
(\omega U \omega)^{\top} \sim \mathfrak{B}(\alpha; c_d,\dots,c_1) .
\end{equation}

Recall that the Cholesky decomposition of a matrix $x\in\pos_d$ is given by $x=u^{\top} u$, where $u$ is upper triangular with positive diagonal entries (such a matrix $u$ is unique).

\begin{proposition}
\label{prop:bartlett}
Let $X$ be a random matrix in $\pos_d$ and let $X=U^{\top}U$ be its Cholesky decomposition.
Let $\alpha,\beta > \frac{d-1}{2}$.
Let $A\sim \mathfrak{B}(\alpha;1,\dots,d)$ and $B\sim \mathfrak{B}(\beta;d,\dots,1)$ be independent.
\begin{enumerate}[(i)]
\item \label{item:CholeskyWishart}
If $X\sim \Wish_{d}(\alpha)$, then $U\stackrel{\mathrm{d}}{=} A$.
\item \label{item:CholeskyInvWishart}
If $X\sim \invWish_d(\beta)$, then $U \stackrel{\mathrm{d}}{=} B^{-1}$.
\item If $X\sim \BetaII_{d}(\alpha,\beta)$, then $U\stackrel{\mathrm{d}}{=}A B^{-1}$.
\end{enumerate}
\end{proposition}
\begin{proof}
\begin{enumerate}[(i)]
\item If $x\in\pos_d$ has Cholesky decomposition $x=u^{\top}u$, we have
\begin{align*}
|x| = \prod_{k=1}^d u_{k,k}^2, \qquad
\tr[x] = \sum_{1\leq i\leq j\leq d} u_{i,j}^2, \qquad
\mu(\diff x) = 2^d \prod_{k=1}^d u_{k,k}^{-k} \prod_{1\leq i\leq j\leq d} \diff u_{i,j} .
\end{align*}
The first two formulas are immediate, while the third formula can be found for example in~\cite[Exercise~1.1.21-(b)]{T}.
By hypothesis, we have
\[
X\sim \Gamma_{d}(\alpha)^{-1}|x|^{\alpha}e^{-\tr[x]}\mu(\diff x).
\]
As $X=U^{\top} U$, it follows easily from the above formulas that
\[
U\sim c
\prod_{k=1}^d \left\{ u_{k,k}^{2\left(\alpha-\frac{k-1}{2}\right)-1} e^{-u_{k,k}^2} \diff u_{k,k} \right\}
\prod_{1\leq i< j\leq d} \left\{e^{-u_{i,j}^2} \diff u_{i,j} \right\} ,
\]
where $c$ is the normalisation constant.
Therefore, all the entries of $U$ are independent, $U_{k,k}^2$ has the gamma distribution with parameter $\alpha-\frac{k-1}{2}$, for $1\leq k\leq d$, and $U_{i,j}$ is Gaussian with mean $0$ and variance $1/2$, for $1\leq i<j\leq d$.
We conclude that $U\sim \mathfrak{B}(\alpha;1,\dots,d)$.
\item By hypothesis, $X^{-1} = U^{-1} (U^{-1})^{\top} \sim \Wish_d(\beta)$.
As the matrix $\omega$ defined in~\eqref{eq:revPermutation} satisfies $\omega \omega^{\top} = I$, by orthogonal invariance of the Wishart distribution we have
\[
X^{-1}
\stackrel{\mathrm{d}}{=} \omega U^{-1} (U^{-1})^{\top} \omega^{\top}
= \omega U^{-1} \omega \omega^{\top} (U^{-1})^{\top} \omega^{\top}
= V^{\top} V ,
\]
where $V:= (\omega U^{-1} \omega)^{\top}$.
As $V$ is an upper triangular matrix with positive diagonal entries, it follows from~\eqref{item:CholeskyWishart} and from the uniqueness of the Cholesky decomposition that $V\sim \mathfrak{B}(\beta;1,\dots,d)$.
By~\eqref{eq:BartlettEquiv}, this implies that $U^{-1} \sim \mathfrak{B}(\beta;d,\dots,1)$, i.e.\ $U \stackrel{\mathrm{d}}{=} B^{-1}$.
\item By Proposition~\ref{relationswishart}-\eqref{it:relationswishart_betaII}, we have $X \stackrel{\mathrm{d}}{=} \mult_{X_{\beta}^{-1}}(X_{\alpha})$, where $X_{\alpha}\sim \Wish_{d}(\alpha)$ and $X_{\beta}\sim \Wish_{d}(\beta)$ are independent, for any choice of $w$.
In particular, we may choose $w$ to be the Cholesky function, so that, using~\eqref{item:CholeskyWishart}, \eqref{item:CholeskyInvWishart}, and the independence of $X_{\alpha}$ and $X_{\beta}$, we have $(X_{\alpha}, X_{\beta}^{-1}) \stackrel{\mathrm{d}}{=} (A^{\top} A, (B^{-1})^{\top} B^{-1})$.
Therefore, we obtain
\[
X
\stackrel{\mathrm{d}}{=} \mult_{X_{\beta}^{-1}}(X_{\alpha})
= (B^{-1})^{\top} A^{\top} A B^{-1}
= (AB^{-1})^{\top} (AB^{-1}) ,
\]
which implies $U\stackrel{\mathrm{d}}{=}A B^{-1}$, by the uniqueness of the Cholesky decomposition. \qedhere
\end{enumerate}
\end{proof}

\section{Lyapunov exponents of random walks on $\pos_d$}
\label{app:lyapunovExponents}

Here we present a method for computing the Lyapunov exponents of $\GL_d$-invariant random walks on $\pos_d$.
To do so, we start from an argument proposed by Newmann~\cite{newmann} and generalise it via the use of Cholesky decompositions.
Next, we specialise this method to the case of random walks with Wishart, inverse Wishart or Beta type II increments.

Throughout this appendix, we denote by $\lambda_{k}(Y)$ the $k$-th largest eigenvalue of a random matrix $Y$ in $\pos_d$, for $1\leq k\leq d$.
The following lemma is a reformulation of an identity that can be found in~\cite[Eq.~(4)]{newmann}.

\begin{lemma}
\label{lemma:Newmann}
Let $Y$ be a nonsingular random matrix such that $Y^{\top} Y$ is $\ort_d$-invariant. 
Denote by $Y^{[k]}$ the (rectangular) matrix of the first $k$ columns of $Y$.
Define $\mu_1,\dots,\mu_d$ by setting
\begin{equation}
\label{eq:Newmann}
\mu_1 +\dots +\mu_k
:= \mathbb{E} \log \left\lvert \big(Y^{[k]}\big)^{\top} Y^{[k]} \right\rvert ,
\qquad 1\leq k\leq d ,
\end{equation}
assuming that the right-hand side are finite for all $k$.
Consider a family $(Y(n), n\geq 1)$ of i.i.d.\ random matrices with the same distribution as $Y$.
Then, for $1\leq k\leq d$, we have
\[
\mu_k =
\lim_{n\to\infty} \frac{1}{n} \log \lambda_k\left((Y(n)\cdots Y(1))^{\top} (Y(n)\cdots Y(1))\right)
\qquad \text{a.s.}
\]
\end{lemma}

\begin{rem}
The equivalence between~\cite[Eq.~(4)]{newmann} and Lemma~\ref{lemma:Newmann} is due to the definition of the usual norm on the $k$-th exterior power of $\mathbb{R}^d$ in terms of a Gram determinant:
\[
\lVert v_1 \wedge \cdots \wedge v_k \rVert^2
= \left\lvert \left(\langle v_i, v_j \rangle\right)_{1\leq i,j\leq k} \right\rvert
= \left\rvert V^{\top} V \right\lvert ,
\]
where $V$ is the $d\times k$ matrix whose columns are the vectors $v_1,\dots, v_k\in \mathbb{R}^d$.
\end{rem}

Let $R=(R(n), n\geq 0)$ be a $\GL_d$-invariant random walk on $\pos_d$ with initial state $R(0)$ and increments $(X(n), n\ge 1)$, where $(X(n), n\ge 1)$ is a family of i.i.d.\ and $\ort_{d}$-invariant random matrices in $\pos_d$ and $R(0)$ is any random matrix in $\pos_d$ independent of $(X(n), n\ge 1)$ (see Section~\ref{subsec:InvRW}).
We will say that the $k$-th \emph{Lyapunov exponent} of the random walk $R$ is the quantity $\mu_k(R)$ defined in the following proposition.

\begin{proposition}
\label{prop:lyapunov}
Let $X(1)=U^{\top}U$, where $U=(U_{i,j})_{1\leq i,j\leq d}$ is the Cholesky decomposition of $X(1)$.
For all $1\leq k\leq d$, we have
\begin{align}
\label{eq:lyapunov}
\mu_k(R)
:= \mathbb{E}\log (U_{k,k}^2)
= \lim_{n\to\infty}\frac{1}{n}\log \lambda_{k}(R(n))
\qquad\text{a.s.},
\end{align}
assuming that all the above expected values are finite.
\end{proposition}
\begin{proof}
Notice first that, if $\xi=(\xi(n), n\geq 0)$ and $\xi'=(\xi'(n), n\geq 0)$ are two real-valued processes with the same law, such that $\xi$ converges a.s., as $n\to\infty$, to a constant $c$, then $\xi'$ also converges a.s.\ to $c$ (this follows from the fact that, since $\xi \stackrel{\mathrm{d}}{=} \xi'$, the probability that $\xi$ converges is the same as the probability that $\xi'$ converges).
Therefore, it suffices to show~\eqref{eq:lyapunov} when $R$ is realised by the $\GL_d$-invariant random walk $\RWmod$ of Proposition~\ref{prop:invRW}, with $w$ being the Cholesky function.
Namely, we let
\begin{align}\label{RCholesky}
R(n)
:= (U(n)\cdots U(0))^{\top} (U(n)\cdots U(0)) , \quad
n\ge 0,
\end{align}
where $X(n)=U(n)^{\top} U(n)$ is the Cholesky decomposition of $X(n)$, for any $n\geq 1$, and $R(0)=U(0)^{\top} U(0)$ is the Cholesky decomposition of $R(0)$.

The next step is to prove the equalities
\begin{equation}
\label{eq:equalitiesU}
\left\lvert \big(U^{[k]}\big)^{\top} U^{[k]} \right\rvert
= \prod_{i=1}^k U_{i,i}^2 ,
\qquad 1\leq k\leq d ,
\end{equation}
where $U^{[k]}$ is the matrix of the first $k$ columns of $U$.
Since $U$ is a $d\times d$ upper triangular matrix, the last $d-k$ rows of $U^{[k]}$ are zero.
Therefore, we have 
\[
\big(U^{[k]}\big)^{\top} U^{[k]}
= \big(U^{[k]}_{[k]}\big)^{\top} U^{[k]}_{[k]} ,
\]
where $U^{[k]}_{[k]}$ is the $k$-th leading principal minor of $U$.
Since $U^{[k]}_{[k]}$ is a $k\times k$ upper triangular matrix with diagonal entries $U_{1,1}, \dots, U_{k,k}$, we then have
\[
\left\lvert \big(U^{[k]}\big)^{\top} U^{[k]} \right\rvert
= \left\lvert \big(U^{[k]}_{[k]}\big)^{\top} U^{[k]}_{[k]} \right\rvert
= \left\lvert U^{[k]}_{[k]} \right\rvert^2
= \prod_{i=1}^k U_{i,i}^2
\]
for all $1\leq k\leq d$, as desired.

By~\eqref{eq:equalitiesU} and the definition~\eqref{eq:lyapunov} of $\mu_k(R)$, we have
\[
\mu_1(R) + \dots + \mu_k(R)
= \mathbb{E} \log \left\lvert \big(U^{[k]}\big)^{\top} U^{[k]} \right\rvert.
\]
It then follows immediately from Lemma~\ref{lemma:Newmann} (taking $Y(n):=U(n)$ for all $n\geq 1$) that~\eqref{eq:lyapunov} holds when $R(0)=U(0)=I$ is the identity matrix.

To prove the claim in the general case where $R(0)$ is an arbitrary random matrix in $\pos_d$ independent of $(X(n), n\ge 1)$, we need an additional simple argument.
Using the standard inequalities $\lambda_d(a) \lambda_k(b) \leq \lambda_k(ab) \leq \lambda_1(a) \lambda_k(b)$ (for $a,b\in\pos_d$), which generalise~\eqref{lambdaproduct}, and the fact that the matrices $ab$ and $ba$ have the same eigenvalues (for $a,b\in\GL_d$), we obtain
\[
\begin{split}
&\lambda_d(R(0)) \, \lambda_k\left((U(n)\cdots U(1))^{\top} (U(n)\cdots U(1))\right)
\leq \lambda_k(R(n)) \\
&\leq \lambda_1(R(0)) \, \lambda_k\left((U(n)\cdots U(1))^{\top} (U(n)\cdots U(1))\right) .
\end{split}
\]
Since both $\lambda_d(R(0))$ and $\lambda_1(R(0))$ are a.s.\ positive, the claim now follows from the case $R(0)=U(0)=I$.
\end{proof}

Let $\psi$ be the \emph{digamma function}, i.e.\ the logarithmic derivative of the gamma function: 
\begin{align}\label{digamma}
\psi(x) := \frac{\diff}{\diff x} \log \Gamma(x) .
\end{align}

\begin{corollary}\label{Lyapunov}
Let $\mu_k(R)$ be the $k$-th Lyapunov exponent of the $\GL_d$-invariant random walk $R$, as defined in~\eqref{eq:lyapunov}.
Let $\alpha,\beta > \frac{d-1}{2}$.
Then,
\begin{enumerate}[(i)]
\item
If $X(1)\sim \Wish_{d}(\alpha)$, then $\mu_k(R) = \psi\left(\alpha-\frac{k-1}{2}\right)$.
\item
If $X(1)\sim \invWish_d(\beta)$, then $\mu_k(R) = -\psi\left(\beta-\frac{d-k}{2}\right)$.
\item
If $X(1)\sim \BetaII_{d}(\alpha,\beta)$, then $\mu_k(R) = \psi\left(\alpha-\frac{k-1}{2}\right)-\psi\left(\beta-\frac{d-k}{2}\right)$.
\end{enumerate}
\end{corollary}
\begin{proof}
It follows immediately from Proposition~\ref{prop:bartlett} that:
\begin{enumerate}[(i)]
\item
If $X(1)\sim \Wish_{d}(\alpha)$, then $U_{k,k}^2 \sim \Wish_1\left(\alpha-\frac{k-1}{2}\right)$.
\item
If $X(1)\sim \invWish_d(\beta)$, then $U_{k,k}^2 \sim \invWish_1\left(\beta-\frac{d-k}{2}\right)$.
\item
If $X(1)\sim \BetaII_{d}(\alpha,\beta)$, then $U_{k,k}^2 \sim \BetaII_1\left(\alpha-\frac{k-1}{2},\beta-\frac{d-k}{2}\right)$.
\end{enumerate}
The claim then follows from the fact that, if $G$ has a univariate Wishart (i.e., gamma) distribution with parameter $\nu$, then $\mathbb{E} \log G = \psi(\nu)$.
\end{proof}

\section{Markov functions}
\label{app:MarkovFunctions}

Let $(S, \mathcal{S})$ and $(S', \mathcal{S}')$ be measurable spaces and $\phi:S\to S'$ be a measurable function.
Consider a time-homogeneous Markov process $X=(X(n),n\ge0)$ with state space $S$ and transition kernel $\Pi$.
Defining $Y(n):=\phi(X(n))$ for all $n\ge0$, we are interested in conditions that guarantee that the transformed process $Y=(Y(n),n\ge0)$ is Markov in \emph{its own} filtration.
The well-known Dynkin criterion~\cite{Dyn} provides conditions under which $Y$ has the Markov property for any possible initial distribution of $X$.
There is also a more subtle criterion, which has been proved at various levels of generality by Kemedy and Snell~\cite{KS}, Kelly~\cite{Kel}, Rogers and Pitman~\cite{RP} and Kurtz~\cite{Kur}, which ensures that $Y$ is Markov only for special initial distributions of $X$.
We review this below, in the setting that is suited to our needs. 

First we need some basic terminology.
Let $\mathfrak{b}{\mathcal{E}}$ denote the set of real-valued bounded measurable functions on a given measurable space $(E, \mathcal{E})$.
A \emph{Markov kernel} from $S'$ to $S$ is a map $N: S'\times\mathcal{S}\to\mathbb{R}$ such that, for each $y\in S'$, $N(y;\cdot)$ is a probability measure on $(S,\mathcal{S})$ and, for each $A\in\mathcal{S}$, $N(\cdot\,; A)$ is an element of $\mathfrak{b}\mathcal{S}'$.
The kernel $N$ can be also viewed as the \emph{Markov operator} that maps $f\in \mathfrak{b}\mathcal{S}$ to $Nf\in \mathfrak{b}\mathcal{S}'$, where
\begin{align}\label{Markovoperator}
N f(y):=\int_{S}N(y;\diff z) f(z)\,,\quad f\in \mathfrak{b}\mathcal{S},\,\,\, y\in S'.
\end{align}
If $(U, \mathcal{U})$ is another measurable space and we consider a Markov kernel $M: S\times\mathcal{U}\to\mathbb{R}$, then
\begin{align}\label{concat}
NM(y;A):=\int_{S}N(y;\diff z)\,M(z;A),\quad y\in S',\,\,\, A\in\mathcal{U},
\end{align}
is again a Markov kernel.
The associated operator is the usual composition of the Markov operators $N$ and $M$.
Therefore, the operation~\eqref{concat} is clearly associative.

We prove the next theorem, by now fairly classical, using a discrete-time version of the argument given for continuous-time Markov processes by Rogers and Pitman in~\cite[Theorem 2]{RP}.

\begin{theorem}\label{MarkovFunctionsT}
Let $\Pi$ be a Markov kernel from $S$ to itself and let $\phi:S\to S'$ be a measurable function.
Assume that $\mathcal{S}'$ contains all the singleton sets $\{y\}$ and that there exist Markov kernels $Q$ from $S'$ to itself and $K$ from $S'$ to $S$ such that
\begin{enumerate}[(i)]
	\item \label{first}
	$K(y;\phi^{-1}\{y\})=1$ for every $y\in S'$;
	\item \label{second}
	$K \, \Pi = Q \, K$.
\end{enumerate}
For any distribution $\eta$ on $S'$, if $X=(X(n),n\ge0)$ is a time-homogeneous Markov process on $S$ with transition kernel $\Pi$ and initial distribution $\eta K$, then $Y=(Y(n),n\ge0)$ defined by $Y(n):=\phi(X(n))$, $n\ge0$, is a time-homogeneous Markov process (in its own filtration) with initial distribution $\eta$ and transition kernel $Q$.
Moreover, for all $f\in\mathfrak{b}\mathcal{S}$ and $n\ge0$, we have
\begin{align}\label{conditionalonY}
\mathbb{E}\left[f(X(n))\Big|Y(0),\dots,Y(n-1),Y(n)\right]=Kf\,(Y(n)), \quad \text{a.s.}
\end{align}
\end{theorem}

\begin{proof}
Note that~\eqref{first} implies
\begin{align}\label{composition}
\int_{S}K(y;\diff x)g(\phi(x))f(x)=g(y)\int_{S}K(y;\diff x)f(x),\qquad y\in S',
\end{align}
for all  $g\in\mathfrak{b}\mathcal{S}'$ and $f\in\mathfrak{b}\mathcal{S}$.
If we define the Markov operator $\Phi:\mathfrak{b}\mathcal{S}'\to\mathfrak{b}\mathcal{S}$ by $\Phi g:=g\circ\phi$ for $g\in\mathfrak{b}\mathcal{S}'$, then~\eqref{composition} may be written as
\begin{align}\label{concatenation}
K(\Phi g)f=g Kf.
\end{align}
In the following (as in equation~\ref{concatenation}), operations should be read from right to left, giving priority to parentheses.
Applying $Q$ to both sides of~\eqref{concatenation} and using the assumption~\eqref{second}, we have
\begin{align}\label{concatenation2}
K\,\Pi(\Phi g) f=Q g Kf.
\end{align}
For  test functions $g_{0},\dots,g_{n}$ in $\mathfrak{b}\mathcal{S}'$ and $f\in\mathfrak{b}\mathcal{S}$, using~\eqref{concatenation} and~\eqref{concatenation2} we obtain
\begin{align}
\notag
\begin{split}
K(\Phi g_{0}) \Pi(\Phi g_{1}) \Pi(\Phi g_{2})\cdots \Pi(\Phi g_{n}) f
&=g_{0} K\,\Pi(\Phi g_{1}) \Pi(\Phi g_{2})\cdots \Pi(\Phi g_{n}) f\\
&=g_{0} Q g_{1} K\,\Pi(\Phi g_{2})\cdots \Pi(\Phi g_{n}) f
\end{split}
\intertext{and, inductively,}
\label{concatenation3}
K(\Phi g_{0}) \Pi(\Phi g_{1}) \cdots \Pi(\Phi g_{n}) f
&=g_{0} Q g_{1} \cdots Q g_{n} K f.
\end{align}

Let $\eta$ be any distribution on $S'$. Applying $\eta$ to both sides of~\eqref{concatenation3} we have
\begin{align}
\label{concatenation4}
\eta K(\Phi g_{0}) \Pi(\Phi g_{1}) \cdots \Pi(\Phi g_{n}) f
&=\eta g_{0} Q g_{1} \cdots Q g_{n} K f.
\end{align}
The hypotheses on the processes $X$ and $Y$, together with the identity~\eqref{concatenation4}, imply that
\begin{align*}
\mathbb{E}\left[g_{0}(Y(0)) g_{1}(Y(1)) \cdots g_{n}(Y(n))f(X(n))\right]=\eta g_{0} Q g_{1} Q g_{2}\cdots Q g_{n} K f.
\end{align*}
Taking $f\equiv 1$, we deduce that $Y$ is a time-homogeneous Markov process with initial distribution $\eta$ and transition kernel $Q$.
For general $f$, the right-hand side of the equation above agrees with
\begin{align*}
\mathbb{E}\left[g_{0}(Y(0)) g_{1}(Y(1)) \cdots g_{n}(Y(n))Kf(Y(n))\right].
\end{align*}
By definition of conditional expectation, we deduce~\eqref{conditionalonY}.
\end{proof}

\begin{rem}\label{oper}
Taking $f\equiv 1$ in~\eqref{concatenation} we have
\begin{align*}
K \Phi=\mathrm{id}_{\mathfrak{b}\mathcal{S}'},
\end{align*}
where $\mathrm{id}_{\mathfrak{b}\mathcal{S}'}$ is the identity operator on $\mathfrak{b}\mathcal{S}'$.
Using this and the assumption~\eqref{second} of Theorem~\ref{MarkovFunctionsT}, it is clear that the Markov kernel $Q$ is uniquely determined by the relation 
\begin{align*}
Q=K \, \Pi \, \Phi.
\end{align*}

\end{rem}

\end{document}